\documentclass[12pt, reqno]{amsart}
\usepackage[utf8]{inputenc}
\usepackage{amsthm, amsmath, amssymb, bm,diffcoeff}

\usepackage{verbatim}
\usepackage{graphicx}
\usepackage[linesnumbered,ruled,vlined]{algorithm2e}
\usepackage{tikz}
\usetikzlibrary{decorations.pathreplacing}
\usepackage{derivative}
\usepackage{mathtools}

\usepackage{xifthen}
\usepackage[colorlinks=true,
linkcolor=blue,citecolor=blue,
urlcolor=blue]{hyperref}
\usepackage[margin=1.2in]{geometry}
\usepackage{extarrows}
\usepackage[shortlabels]{enumitem}

\usetikzlibrary{calc,shapes, backgrounds}
\allowdisplaybreaks

\usepackage{tikz}   % 放在导言区
\usepackage{float}
\usetikzlibrary{matrix,arrows} 

\newtheorem{theorem}{Theorem}[section]
\newtheorem{lemma}[theorem]{Lemma}
\newtheorem{corollary}{Corollary}
\newtheorem{definition}{Definition}

\newtheorem{remark}{Remark}

\newtheorem{proposition}{Proposition}
\newtheorem{example}{Example}

\newtheorem{question}[theorem]{Question}

\usepackage{xcolor}

\newcommand{\norm}[1]{\left\lVert#1\right\rVert}

\usepackage{booktabs}
\usepackage{multirow}
\usepackage{tabularx}
\usepackage{makecell}
\usepackage[table]{xcolor}
\definecolor{lightgray}{gray}{0.95}
\usepackage{standalone}
\usepackage{amsmath}
\usepackage{graphicx}

\title[Discrete Einstein metrics on trees]{Discrete Einstein metrics on trees}

\author{
Shuliang Bai}
\thanks{Beijing Yanqi Lake Institute of Mathematical Sciences and Applications, China. 
Email address: baishuliang@bimsa.cn, corresponding author}
\author{Haoxuan Cheng}
\thanks{School of Mathematical Sciences, Fudan University, Shanghai 200433, China. Email address: hxcheng25@m.fudan.edu.cn} 
\author{Bobo Hua}
\thanks{School of Mathematical Sciences, LMNS, Fudan University, Shanghai
200433, China.  
Email address: bobohua@fudan.edu.cn.}

\date{\today}

\begin{document}

\begin{abstract}

We establish the existence and uniqueness of discrete Einstein metrics on trees under Lin-Lu-Yau Ricci curvature using Perron-Frobenius theory. We derive the quantitative bounds for the largest eigenvalue of the associated Ricci matrix, including sharp asymptotics for regular trees and a universal upper bound depending only on the maximum
degree. Turning to structural properties, notably, the existence of a positive-curvature Einstein metric implies the tree must be a caterpillar. Furthermore, these positive-curvature metrics exhibit radial monotonicity, with edge weights decreasing strictly away from the maximal edge.

\end{abstract}
\maketitle
%\keywords{Discrete Einstein metric; Ricci matrix; Largest eigenvalue, Monotonicity}

%\noindent \textbf{Mathematics Subject Classification (2020):} 53C21,  05C50, 53C25, 05C05, 31C20, 05C76

\section{Introduction}

A primary goal in Riemannian geometry is determining if a manifold is Einstein, meaning it possesses constant Ricci curvature, i.e.

$$Ric(g) = \kappa g,$$
where $g$ is the Riemannian metric and $\kappa$ is a constant. A complete answer for surfaces was given by the uniformization theorem.
The literature offers several profound results for higher dimensional cases \cite{yau1978ricci,besse1987einstein}.  {Recent progress} often involves the Ricci flow, where the Einstein condition emerges as a fixed point of the normalized Ricci flow \cite{hami1982,perelman2002,perelman2003a,brendle2009differentiable,RFTA1}. Despite these insights, the general question of which smooth manifolds admit Einstein metrics remains one of the most difficult challenges in modern geometric analysis.

Notions of discrete curvature on graphs have been introduced to capture geometric features via transport and diffusion processes, most notably through Ollivier Ricci curvature and the Lin–Lu–Yau formulation \cite{Ollivier,LLY,MW}. These approaches interpret curvature as inducing local redistribution mechanisms on networks; see e.g. \cite{Paeng2012,BauerJostLiu2012,JostLiu2014,BCLMP2018,CKKLMP2020,KuzuyoshiTaiki2020,CKLLLY21,HLYZ2021,blhy,bai2021ricci,Muench2023,HLTK2023,HuaMuench2025} for recent developments. In this paper, we discuss existence and uniqueness results of discrete Einstein metrics on trees.

Let $G=(V,E,w)$ be a finite weighted graph with edge weight $w:E\to \mathbb{R}_+.$ We denote by $\kappa_{xy}$ the Ricci curvature in the sense of Lin-Lu-Yau, a discrete analogue of Ricci curvature, for any edge $\{x,y\}$ on $G;$ see Definition~\ref{def:lly}. A weighted graph is called \emph{discrete Einstein} (in the sense of Lin-Lu-Yau) if $\kappa_{xy}=\kappa,$ $\kappa\in \mathbb{R}$, for any edge $\{x,y\}.$ Notably, the Ricci curvature has a closed-form for trees; see \eqref{eq:lly-1}.

The main problems are the following.
\begin{question}\label{qu:1}
Does a graph admit a discrete Einstein metric? How about the uniqueness of discrete Einstein metrics?
\end{question}

In this paper, we study the problem for the class of trees.  We introduce a key {\it Ricci matrix} for a tree. 
 \begin{definition}%[  Ricci matrix from Ricci Flow on Trees]
\label{def:evolution_matrix}
For a finite tree $ T = (V, E),$ a matrix $R_T \in \mathbb{R}^{|E| \times |E|} $ indexed by the edges of $ T $ is defined as,  
for an edge $ e = \{x, y\} \in E $  and $ e' \in E $, 
\begin{align}\label{eq:matrix}
(R_T)_{e,e'} =
\begin{cases}
-\left( \dfrac{1}{d_x} + \dfrac{1}{d_y} \right), & \text{if } e = e', \\[8pt]
\dfrac{1}{d_x}, & \text{if } e \cap e' = \{x\}, \\[8pt]
0, & \text{otherwise},
\end{cases}
    \end{align}
where $ d_x $  denotes the degree of a vertex $ x .$

\end{definition}
\begin{remark}\label{rem:schr}
    The Ricci matrix can be regarded as a weighted Schr\"odinger operator on the line graph of $T,$ $R_T=\Delta-V$ where $\Delta$ is a weighted Laplacian on the line graph with weights defined in \eqref{eq:matrix} and $V$ is a diagonal matrix. Formally, \[
\Delta_{e,e'} =
\begin{cases}
\dfrac{1}{d_x} + \dfrac{1}{d_y} - 2, & \text{if } e = e' = \{x,y\}, \\[10pt]
\dfrac{1}{d_x}, & \text{if } e \cap e' = \{x\}, \\[8pt]
0, & \text{otherwise},
\end{cases}
\qquad
V_{e,e'} = 
\begin{cases}
\frac{2}{d_x} + \frac{2}{d_y} - 2, & e = e', \\
0, & e \neq e'.
\end{cases}
\]
 One could study the eigenvalue problems using the theory of Schr\"odinger operators.

\end{remark}

The discrete Ricci flow for the Lin-Lu-Yau curvature was introduced by \cite{bailin}, providing fundamental conditions for the existence and uniqueness of the flow's solution. Building upon this foundation, \cite{Baihua} extended the analysis to finite trees, investigating the long-time dynamics of edge weights and curvatures under the continuous-time evolution. Remarkably, the Ricci matrix $R_T$ appears in a simplified version of the
discrete Ricci flow, the linear ODE
\[
\frac{d}{dt} w(t) = R_T\, w(t),
\]
where $w(t)$ denotes the edge weights. Its solution is $w(t)=e^{tR_T}w(0)$,
and for any positive initial weight the normalized solution
$w(t)/\norm{w(t)}$ converges to the Perron eigenvector of $R_T$, which by
Theorem~\ref{thm:perron_structure} is precisely a discrete Einstein metric.

We give an affirmative answer to Question~\ref{qu:1} via the  Ricci matrix $R_T$ and the Perron–Frobenius theory; see e.g. \cite{BermanPlemmons1979,HM}. 

\begin{theorem}[Spectral characterization of constant Ricci curvature on trees]
\label{thm:perron_structure}
Let $T=(V,E)$ be a finite tree, and let $R_T \in \mathbb{R}^{|E|\times |E|}$ be its Ricci matrix.

\begin{enumerate}
    \item (\textbf{Perron eigenpair})  
    The largest eigenvalue $\lambda$ of $R_T$ is simple, and there exists a corresponding eigenvector $w>0$ which is unique up to scaling.

    \item (\textbf{Sign structure})  
    The Perron eigenvector is the only eigenvector that is entrywise positive. All other eigenvectors must change sign.

    \item (\textbf{Einstein curvature characterization})  
     A positive edge weight $w$ is discrete Einstein of curvature $\kappa$
      if and only if
    \[
    R_T w = \lambda_{\max} w, \quad \kappa = -\lambda_{\max}.
    \]

    \item (\textbf{Uniqueness of Einstein curvature metric})  
    Up to scaling, the Perron eigenvector is the only positive weight function on $T$ that induces constant Ricci curvature.
\end{enumerate}
\end{theorem}

We also establish quantitative bounds for the Perron eigenvalue $\lambda_{\max}(R_T)$. In Theorem~\ref{thm:lower bound}, we prove that for any non-star tree $T$,
\[
\min_{\mathrm{internal\ edge \ }e}(-\widetilde{V}_{e,e})\le \lambda_{\max}(R_T)\le  \max_{e\in E}(-V_{e,e}),
\]
where $V$ is the diagonal potential matrix from the Schr\"odinger operator representation. For balls $T_{d,L}$ in $d$-regular trees of depth $L$, we obtain sharp asymptotic estimates (Propositions~\ref{prop:td1} and \ref{prop:td2}):
\[
\lim_{L \to \infty} \lambda_{\max}(R_{T_{d, L}}) = 1 - \frac{4}{d} + \frac{2\sqrt{d-1}}{d}.
\]

Moreover, we obtain a universal upper bound for $\lambda_{\max}(R_T)$ that depends only on 
the maximum degree of the tree.

\begin{theorem}\label{thm:degree-bound}
Let $T$ be a finite tree with maximum degree $\mathcal D = \max_{v\in V} d_v$. Then
\[
\lambda_{\max}(R_T)\le
\begin{cases}
-2, & \mathcal D=1,\\[4pt]
1-\dfrac{4}{\mathcal D}+\dfrac{2\sqrt{\mathcal D-1}}{\mathcal D}, & 2\le \mathcal D\le 18,\\[8pt]
\dfrac{15+6\sqrt2}{19}, & \mathcal D\ge 19.
\end{cases}
\]
\end{theorem}

This bound is at $\mathcal D=1$ and 
attained asymptotically by regular trees when $2\le\mathcal D\le 18$; for $\mathcal D\ge 19$, 
the constant $\frac{15+6\sqrt2}{19}$ is the maximum of $F(x)=1-4/x+2\sqrt{x-1}/x$ 
over integers $x\ge 2$, attained at $x=19$.

\medskip
Next, we study how $\lambda_{\max}(R_T)$ changes under attachment.
Proposition~\ref{pro:attachment} gives a single threshold $\Lambda(d,k)$: attaching a
forest of $k$ trees at a vertex of degree $d$ cannot decrease $\lambda_{\max}$ as long as
$\lambda_{\max}(R_T)\le\Lambda(d,k)$. In particular, $\lambda_{\max}$ never decreases at a
vertex of degree $\le 2$, and strictly increases at \emph{any} vertex when
$\lambda_{\max}(R_T)\le 0$. Combined with the critical subtree $S_3^2$
(Definition~\ref{def:S32}), this monotonicity shows that every tree admitting a
positive-curvature Einstein metric is a {\it caterpillar}, a tree whose non-leaf edges
form a single path (Definition~\ref{def:caterpillar}).

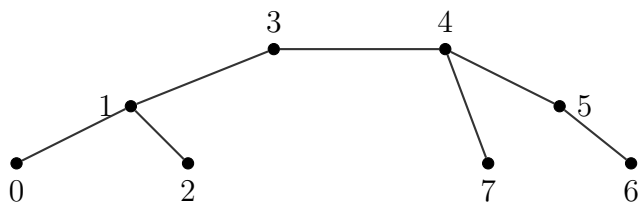
\begin{figure}[H]
\centering
\begin{tikzpicture}[
    scale=1.9, % 整体微缩以适应定理下方的空间
    nodestyle/.style={circle, draw, fill=black, inner sep=1.5pt},
    labelstyle/.style={font=\scriptsize}
]

    % --- 节点坐标布局 (扁平化处理) ---
    % 中心脊柱部分
    \node[nodestyle, label=above:{  3}] (n3) at (0,0) {};
    \node[nodestyle, label=above:{  4}] (n4) at (1.2,0) {};
    
    % 左侧分叉 (节点1连接0, 2, 3)
    \node[nodestyle, label=left:{  1}] (n1) at (-1.0, -0.4) {};
    \node[nodestyle, label=below:{  0}] (n0) at (-1.8, -0.8) {};
    \node[nodestyle, label=below:{  2}] (n2) at (-0.6, -0.8) {};
    
    % 右侧分叉 (节点4连接5, 7; 节点5连接6)
    \node[nodestyle, label=right:{  5}] (n5) at (2.0, -0.4) {};
    \node[nodestyle, label=below:{  6}] (n6) at (2.5, -0.8) {};
    \node[nodestyle, label=below:{  7}] (n7) at (1.5, -0.8) {};

    % --- 统一线宽的边绘制 ---
    \begin{scope}[thick, black!80]
        \draw (n3) -- (n4);
        \draw (n1) -- (n3);
        \draw (n4) -- (n5);
        \draw (n4) -- (n7);
        \draw (n1) -- (n0);
        \draw (n1) -- (n2);
        \draw (n5) -- (n6);
    \end{scope}

\end{tikzpicture}
\caption{The topology of a caterpillar tree.}
\label{fig:tree_topology_uniform}
\end{figure}

\begin{theorem}
\label{the:lambda-negative-caterpillar}
If a tree $T$ admits a discrete Einstein metric with positive curvature, then $T$ is a caterpillar tree.
\end{theorem}
\begin{remark}

The converse is false: there exist caterpillar trees with negative-curvature Einstein metrics; see Examples~\ref{ex:double-star}.
\end{remark}

Moreover, we prove the strict monotonicity of the Einstein metric, the Perron vector, in the positive-curvature case, which is decreasing from the edge with maximal weight along radial direction towards the leaves; see e.g. Figure~\ref{fig:perron_compact}.
\begin{theorem}\label{coro:uniquemax}
    Let $T$ be a nonstar-tree with a positive-curvature Einstein metric $w$.
Then there are at most two edges attaining maximal weight $w.$ 
Moreover, if two such edges exist, they must share a common vertex of degree $2.$
Let $e_1=\mathrm{max}_{e\in E} w_e.$ For any leaf edge $f$ and any simple path $e_1\sim e_2\sim \cdots \sim e_k\sim \cdots\sim e_N=f,$ then $w_{e_1}\geq w_{e_2},$ and $w_{e_i}>w_{e_{i+1}}$ for $2\leq i\leq N-1.$
\end{theorem}

\begin{figure}[H] % 使用 [H] 强制紧跟文字
\centering
\begin{tikzpicture}[
    scale=2, % 略微缩小整体比例
    nodestyle/.style={circle, draw, fill=black, inner sep=1.5pt},
    weightstyle/.style={font=\scriptsize, color=blue!80!black}
]

    % --- 核心节点 (压缩水平间距至 1.2) ---
    \node[nodestyle, label=above:{\small 3}] (n3) at (0,0) {};
    \node[nodestyle, label=above:{\small 4}] (n4) at (1.2,0) {};
    
    % --- 左侧分支 (压缩水平间距) ---
    \node[nodestyle, label=left:{\small 1}] (n1) at (-0.8, -0.6) {};
    \node[nodestyle, label=below:{\small 0}] (n0) at (-1.4, -1.2) {};
    \node[nodestyle, label=below:{\small 2}] (n2) at (-0.2, -1.2) {};
    
    % --- 右侧分支 (压缩水平间距) ---
    \node[nodestyle, label=right:{\small 5}] (n5) at (2.0, -0.6) {};
    \node[nodestyle, label=below:{\small 6}] (n6) at (2.6, -1.2) {};
    \node[nodestyle, label=below:{\small 7}] (n7) at (1.4, -1.2) {};

    % --- 绘制边并标注权重 ---
    % 核心边
    \draw[line width=2.5pt, black!80] (n3) -- node[above, weightstyle] {1.00} (n4);
    
    % 次级边
    \draw[line width=1.8pt, black!70] (n1) -- node[above left, pos=0.4, weightstyle, xshift=2pt] {0.85} (n3);
    \draw[line width=1.6pt, black!70] (n4) -- node[above right, pos=0.4, weightstyle, xshift=-2pt] {0.74} (n5);
    
    % 末端边
    \draw[line width=1.0pt, black!60] (n4) -- node[right, pos=0.6, weightstyle] {0.44} (n7);
    \draw[line width=0.6pt, black!50] (n1) -- node[left, pos=0.7, weightstyle] {0.29} (n0);
    \draw[line width=0.6pt, black!50] (n1) -- node[right, pos=0.7, weightstyle] {0.29} (n2);
    \draw[line width=0.5pt, black!50] (n5) -- node[right, pos=0.7, weightstyle] {0.25} (n6);

\end{tikzpicture}
\caption{The Einstein metric on a caterpillar tree with $\kappa \approx 0.0168$.}
\label{fig:perron_compact}
\end{figure}
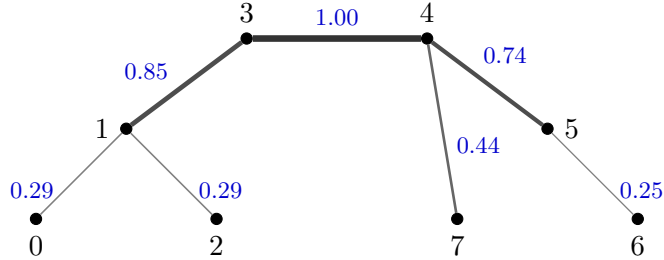

\begin{remark}
It is a well-established consequence of Perron--Frobenius theory for acyclic matrices that the Perron vector of a tree's adjacency matrix attains a unique maximum and decreases strictly along any simple path emanating from that maximum \cite{Fiedler1975Acyclic}. 
\end{remark}
We further analyze the local behavior of the Perron vector. Corollary~\ref{cor:leaf-local} shows that at any vertex, incident leaf edges share equal weight and are strictly lighter than incident internal edges. For $\lambda \le 0$, the global minimum always occurs at a leaf edge (Corollary~\ref{cor:leaf-minimality} and Proposition~\ref{pro:leaf-edge-min}), while for $\lambda > 0$, the global maximum is confined to internal edges (Proposition~\ref{pro:schrodingemaximalweight}). Interestingly, the minimum for $\lambda > 0$ can lie on an internal edge, as shown by the counterexample $D^{29}_{4,4}$. 

Section~\ref{sec:examples} illustrates the spectral phase transition through explicit families. Example~\ref{ex:cospectral-pair} gives two non-isomorphic trees on $17$ vertices that are cospectral with respect to $R_T$, showing that the full spectrum does not uniquely determine the tree structure. Examples~\ref{ex:double-star}, \ref{ex:Tmk} exhibit parameter regimes where $\lambda_{\max}$ transitions from negative to zero to positive, revealing the interplay between path-like propagation and branching.

Finally, the Appendix \ref{sec:appendix} provides counterexamples showing that $\lambda_{\max}(R_T)$ is not monotone under edge subdivision or leaf attachment at high-degree vertices, highlighting the subtle dependence of the eigenvalue on global tree topology.
\medskip

\textbf{Organization.} 
Section~\ref{notation} introduces basic notations and proves the spectral characterization 
(Theorem~\ref{thm:perron_structure}). 
Section~\ref{bounds} establishes quantitative bounds for $\lambda_{\max}(R_T)$. 
Section~\ref{sec:eigenvalueinc} studies eigenvalue monotonicity under leaf attachment and 
proves the caterpillar classification. 
Section~\ref{perro<0} analyzes the positive curvature regime and proves radial monotonicity of edge weights. 
Section~\ref{sec:leafminmax} studies the local leaf weights and locates the extremal edges.
Appedix~\ref{sec:appendix} presents examples and counterexamples illustrating the spectral 
phase transition. 
Appedix~\ref{sec:openproblems} concludes with open problems.

\section{Notation and setup} \label{notation}
Throughout this paper, we consider only finite, simple, undirected trees.
A \emph{tree} $T = (V, E)$ is a connected, acyclic graph, where $V$ denotes the set of vertices and $E \subseteq \binom{V}{2}$ denotes the set of edges. For a tree with $n$ vertices, we have $|E| = n-1$.
For a vertex $v \in V$, its \emph{degree} $d_v$ or $d_v(T)$ is the number of edges incident to $v$. A vertex of degree $1$ is called a \emph{leaf} (or pendant vertex). An edge incident to a leaf is called a \emph{leaf edge}. Vertices with degree $\ge 2$ are called \emph{internal vertices}, and edges whose both endpoints are internal are called \emph{internal edges}.

\begin{definition}[Star tree]
A \emph{star} is a tree with exactly one internal vertex (the \emph{center}) and all other vertices being leaves. Equivalently, a star is a complete bipartite graph $K_{1,m}$ for some $m \ge 1$.
\end{definition}

\begin{definition}[Caterpillar Trees]
\label{def:caterpillar}
A \emph{caterpillar tree} is a tree $T=(V,E)$ that contains a path 
$P = (v_1,v_2,\dots,v_{\ell})$ with $\ell\ge 1$, called the \emph{spine}, such that 
\[
V = V(P) \cup L,
\]
where $L$ is the set of leaves of $T$. More precisely:
\begin{itemize}
    \item If $v_i$ is an internal vertex of $P$ (i.e., $2 \leq i \leq \ell-1$), then $v_i$ is adjacent to exactly $d_{v_i}-2$ leaves in $L$.
    \item If $v_i$ is an endpoint of $P$ (i.e., $i=1$ or $i=\ell$), then $v_i$ is adjacent to exactly $d_{v_i}-1$ leaves in $L$.
\end{itemize}

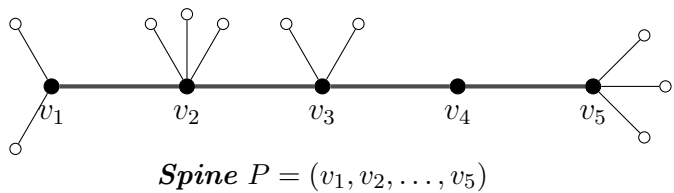
\begin{figure}[htbp]
\centering
\begin{tikzpicture}[
    scale=1.2,
    % 脊柱节点样式：黑色实心圆
    spine/.style={circle, draw, fill=black, inner sep=2pt},
    % 叶子节点样式：空心圆
    leaf/.style={circle, draw, fill=white, inner sep=1.5pt},
    % 脊柱连线：加粗以突出骨架
    spine-edge/.style={ultra thick, black!70}
]

    % --- 绘制脊柱 (Spine) ---
    \node[spine, label=below:{$v_1$}] (v1) at (0,0) {};
    \node[spine, label=below:{$v_2$}] (v2) at (1.5,0) {};
    \node[spine, label=below:{$v_3$}] (v3) at (3,0) {};
    \node[spine, label=below:{$v_4$}] (v4) at (4.5,0) {};
    \node[spine, label=below:{$v_5$}] (v5) at (6,0) {};

    \draw[spine-edge] (v1) -- (v2) -- (v3) -- (v4) -- (v5);

    % --- 绘制附属叶子 (Leaves) ---
    
    % v1 的叶子 (例如 d_v1 = 3, 连接 2 个叶子)
    \foreach \a in {120, 240} {
        \node[leaf] (l1\a) at ($(v1)+(\a:0.8)$) {};
        \draw (v1) -- (l1\a);
    }

    % v2 的叶子 (例如 d_v2 = 5, 连接 3 个叶子)
    \foreach \a in {60, 90, 120} {
        \node[leaf] (l2\a) at ($(v2)+(\a:0.8)$) {};
        \draw (v2) -- (l2\a);
    }

    % v3 的叶子 (例如 d_v3 = 4, 连接 2 个叶子，都在上面)
    \foreach \a in {60, 120} {
        \node[leaf] (l3\a) at ($(v3)+(\a:0.8)$) {};
        \draw (v3) -- (l3\a);
    }

    % v4 无额外叶子 (d_v4 = 2)

    % v5 的叶子 (端点，例如 d_v5 = 4, 连接 3 个叶子)
    \foreach \a in {45, 0, -45} {
        \node[leaf] (l5\a) at ($(v5)+(\a:0.8)$) {};
        \draw (v5) -- (l5\a);
    }

    % 标注
    \node[draw=none, fill=none] at (3, -1) {\small \textbf{Spine} $P = (v_1, v_2, \dots, v_5)$};

\end{tikzpicture}
\caption{An example of a caterpillar tree.}
\label{fig:caterpillar_example}
\end{figure}

\end{definition}
In particular, every path graph $P_n$ ($n \ge 1$) is a caterpillar, where the spine is the entire path; every star is a caterpillar where the spine is a path of length zero.

\subsection{Origin of the  Ricci matrix:} 
The entries of the matrix $R$ are derived from the Ricci flow based on the type of Ollivier's Ricci curvature on edges of trees.

\begin{definition}\label{def:lly} \cite{LLY}
Given local probability distribution $\mu_x^\alpha$ for every vertex $x$, the \emph{$\alpha$-Ricci curvature} between two adjacent vertices $x \sim y$ is defined as
\[
\kappa_{\alpha}(x, y):= 1 - \frac{W(\mu_x^{\alpha}, \mu_y^{\alpha})}{d(x, y)},
\]
where $W(\mu_x^{\alpha}, \mu_y^{\alpha})$ denotes the transportation distance between $\mu_x^{\alpha}$ and $\mu_y^{\alpha}$, and $d(x, y)$ is the weighted distance between $x$ and $y$.

The \emph{(Lin–Lu–Yau) Ricci curvature} is then defined as the limit
\[
\kappa_{xy} := \lim_{\alpha \to 1} \frac{\kappa_{\alpha}(x, y)}{1 - \alpha}.
\]
\end{definition}

 We consider the probability measure $\mu_x^\alpha$ supported in the neighbourhood of $x$ as:
\[
\mu_x^\alpha(z) = 
\begin{cases}
\alpha, & \text{if } z = x, \\[4pt]
(1 - \alpha)\dfrac{1}{d_x}, & \text{if } z \sim x, \\[4pt]
0, & \text{otherwise}.
\end{cases}
\]
For a tree $T$, this curvature $\kappa_{xy}$ admits a simple closed-form expression in terms of the edge weights and vertex degrees.  
By solving the optimal transport problem between the measures $\mu_x^\alpha$ and $\mu_y^\alpha$ and taking the limit $\alpha \to 1$, one obtains (see \cite{Baihua} for details):
\begin{equation}\label{eq:lly-1}
    \kappa_{xy} = - \left( \frac{S_x - 2 w_{xy}}{w_{xy} d_x} + \frac{S_y - 2 w_{xy}}{w_{xy} d_y} \right), 
\quad \text{where } S_v := \sum_{u \sim v} w_{vu}.
\end{equation}

\begin{definition}
Let $T=(V,E)$ be a finite tree with positive edge weights $w:E\to\mathbb{R}_{>0}$. 
The weighted tree $(T,w)$ is called \emph{discrete Einstein} if the Lin-Lu-Yau Ricci curvature $\kappa_{xy}$ (\ref{eq:lly-1}) is independent 
of the edge $xy$, i.e., $\kappa_{xy}\equiv\kappa$ for some constant $\kappa\in\mathbb{R}$.
\end{definition}
\subsection{Perron Structure of the Ricci matrix}
In this part, we give the proof of Theorem~\ref{thm:perron_structure}. 

\begin{proof}[Proof of Theorem~\ref{thm:perron_structure}]
We first establish the Perron structure. Define
\[
\tilde R_T := R_T + \alpha I,
\quad 
\alpha := \max_{e=xy} \left(\frac{1}{d_x} + \frac{1}{d_y}\right),
\]
so that $\tilde R_T \ge 0$ entrywise. 

Since $T$ is connected, its line graph is also connected, and hence $\tilde R_T$ is irreducible. By the Perron--Frobenius theorem, $\tilde R_T$ has a simple largest eigenvalue with a strictly positive eigenvector. Shifting back by $\alpha$ preserves simplicity and positivity, proving (1).

Statement (2) follows directly from Perron--Frobenius theory.

For (3), recall that the Ricci curvature of an edge $e=\{u,v\}$ is
\[
\kappa_e = -\left(\frac{S_u - 2w_e}{w_e d_u} + \frac{S_v - 2w_e}{w_e d_v}\right).
\]
If $\kappa_e \equiv \kappa$ is constant, multiplying by $w_e$ yields
\[
-\kappa w_e = \frac{S_u - 2w_e}{d_u} + \frac{S_v - 2w_e}{d_v},
\]
which is exactly $R_T w = \lambda w$ with $\lambda = -\kappa$. Thus $w$ is an eigenvector. Since $w>0$, it must be the Perron eigenvector.

Finally, (4) follows from the uniqueness (up to scaling) of the positive eigenvector.
\end{proof}

\subsection{Examples}
\begin{example}[Perron eigenpair of $S_n$]
\label{prop:star-eigenpair}
For the star graph $S_n$ with $n\ge 2$ vertices, the Perron eigenvalue $\lambda_{S_n}$ of the Ricci  matrix $R$ is given by
\[
\lambda_{S_n} = -\frac{2}{n-1} = -\frac{2}{|E|} < 0.
\]

The corresponding Perron eigenvector $w = (w_1,\dots,w_{n-1})^T$ has all components equal:
\[
w_i = 1 \quad \text{for all } i = 1,\dots,n-1 \quad (\text{up to scaling}).
\]
\end{example}

\begin{example}[Perron eigenpair of $P_n$]
\label{prop:path-eigenpair}
For the path graph $P_n$ with $n\ge 3$ vertices, the Perron eigenvalue $\lambda_{P_n}$ of the Ricci matrix $R$ is given by
\[
\lambda_{P_n} = -1 + \cos\left(\frac{\pi}{n-1}\right)<0.
\]

In particular:
\begin{itemize}
\item For $n=3$ ($P_3$), the Perron eigenvalue is
$
\lambda=-1,
$
with eigenvector $(1,1)^T$.

\item For $n=4$ ($P_4$), the Perron eigenvalue is
$
\lambda=-\frac12,
$
with eigenvector proportional to
$
(1,\sqrt2,1)^T.
$

\item As $n\to\infty$, the largest eigenvalue satisfies
$
\lambda \to 0^{-}.
$
\end{itemize}
\end{example}

\section{Bounds of $\lambda_{\max}(R_T)$ and regular tree}\label{bounds}
This section establishes quantitative estimates for the largest eigenvalue $\lambda_{\max}(R_T)$ 
of the Ricci matrix. We first give general bounds valid for any tree, then specialize to 
$d$-regular trees to obtain sharp asymptotics, and finally derive a universal upper bound 
depending only on the maximum degree.

\subsection{General bounds}
\label{subsec:general-bounds}
Recall that in Remark~\ref{rem:schr}  $R_T=\Delta-V$ where $V$ is a diagonal matrix with $V_{e,e}=\frac{2}{d_x}+\frac{2}{d_y}-2$ for $e=\{x,y\}.$
For any internal edge $e=\{x,y\},$ i.e. $d_x, d_y\ge 2$, we define $\widetilde{V}_{e,e}=\frac{2+d_x^0}{d_x}+\frac{2+d_y^0}{d_y}-2,$ where $d_x^0$ is the number of leaf edges incident to $x.$ We have the following estimate of $\lambda_{\max}.$
\begin{theorem}\label{thm:lower bound}
   For a non-star tree $T,$
   $$\min_{\mathrm{internal\ edge \ }e}(-\widetilde{V}_{e,e})\le \lambda_{\max}(R_T)\le  \max_{e\in E}(-V_{e,e}).$$ 
\end{theorem}
\begin{proof}
Note that $$\lambda_{\max}(R_T)\leq \lambda_{\max}(\Delta)+\lambda_{\max}(-V).$$ This proves the upper bound estimate of $\lambda_{\max}(R_T)$. 

For the lower bound estimate.
    Let $x$ be the indicator vector on the set of internal edges, i.e. $x_e=1$ for any internal edge $e$, and $x_e=0,$ otherwise.
    Note that $$x^\top R_T x\geq \min_{\mathrm{internal\ edge \ }e}(-\widetilde{V}_{e,e}) |x|^2.$$ This proves the result by the variational principle.
\end{proof}

\begin{corollary}
    Let $T$ be a tree such that internal vertices are of degree $d,$ $d\geq 4,$ and for each internal edge $e={x,y},$ $d_x^0+d_y^0\leq d-1.$ Then 
    $$\frac{d-3}{d}\leq \lambda_{\max}\leq 2-\frac{4}{d}.$$
\end{corollary}

\begin{proof}
    One easily verifies that for each internal edge
    $$\widetilde{V}_{e,e}=\frac{4+d_x^0+d_y^0-2d}{d}\leq \frac{3-d}{d}.$$ The corollary follows from  Theorem~\ref{thm:lower bound}.
\end{proof}

\subsection{Bounds for $d$-regular trees}
\label{subsec:regular-trees}

Let $d_c$ be the combinatorial distance on a graph, i.e. $d_c(x,y):=\inf\{n:\exists x=x_0\sim x_1\sim \cdots \sim x_n=y\}.$ Denote by $B_L(x):=\{y\in V: d_c(x,y)\leq L\}$ the ball of radius $L$ centered at $x.$ For the $d$-regular tree $\mathbb{T}_d$, denote by $T_{d,L}$ be the induced subgraph on $B_L(o)$ where $o$ is the root.
 \begin{remark}\label{rem:bigtree}
    This produces many examples with positive $\lambda_{\max},$ e.g. a large class of subtrees (such as $T_{d,L}$).
\end{remark}
\begin{figure}[htbp]
\begin{center}
\begin{tikzpicture}[
    scale=1.2,
    v/.style={circle, draw, fill=black, inner sep=1.5pt},
    l/.style={circle, draw, fill=gray!20, inner sep=1.5pt},
    ie/.style={thick, black!60},
    le/.style={thick, blue!60, dashed},
    decorate
]

% 根
\node[v] (o) at (0,0) {};

% ========== 左子树 (u1) ==========
\node[v] (u1) at (-3.0,-1.2) {};
\node[v] (v1) at (-4.0,-2.4) {};
\node[v] (v2) at (-2.0,-2.4) {};
\node[l] (l1) at (-4.4,-3.6) {};
\node[l] (l2) at (-3.6,-3.6) {};
\node[l] (l3) at (-2.4,-3.6) {};
\node[l] (l4) at (-1.6,-3.6) {};

\draw[ie] (o) -- (u1);
\draw[ie] (u1) -- (v1); \draw[ie] (u1) -- (v2);
\draw[le] (v1) -- (l1); \draw[le] (v1) -- (l2);
\draw[le] (v2) -- (l3); \draw[le] (v2) -- (l4);

% ========== 中间子树 (u2) ==========
\node[v] (u2) at (0,-1.2) {};
\node[v] (v3) at (-0.8,-2.4) {};
\node[v] (v4) at (0.8,-2.4) {};
\node[l] (l5) at (-1.2,-3.6) {};
\node[l] (l6) at (-0.4,-3.6) {};
\node[l] (l7) at (0.4,-3.6) {};
\node[l] (l8) at (1.2,-3.6) {};

\draw[ie] (o) -- (u2);
\draw[ie] (u2) -- (v3); \draw[ie] (u2) -- (v4);
\draw[le] (v3) -- (l5); \draw[le] (v3) -- (l6);
\draw[le] (v4) -- (l7); \draw[le] (v4) -- (l8);

% ========== 右子树 (u3) ==========
\node[v] (u3) at (3.0,-1.2) {};
\node[v] (v5) at (2.0,-2.4) {};
\node[v] (v6) at (4.0,-2.4) {};
\node[l] (l9)  at (1.6,-3.6) {};
\node[l] (l10) at (2.4,-3.6) {};
\node[l] (l11) at (3.6,-3.6) {};
\node[l] (l12) at (4.4,-3.6) {};

\draw[ie] (o) -- (u3);
\draw[ie] (u3) -- (v5); \draw[ie] (u3) -- (v6);
\draw[le] (v5) -- (l9);  \draw[le] (v5) -- (l10);
\draw[le] (v6) -- (l11); \draw[le] (v6) -- (l12);

% 分层标注
\draw[decorate,decoration={brace,amplitude=6pt}]
    (-5.2,0.3) -- (-5.2,-0.3) node[midway, left, xshift=-8pt] {depth 0};
\draw[decorate,decoration={brace,amplitude=6pt}]
    (-5.2,-0.9) -- (-5.2,-1.5) node[midway, left, xshift=-8pt] {depth 1};
\draw[decorate,decoration={brace,amplitude=6pt}]
    (-5.2,-2.1) -- (-5.2,-2.7) node[midway, left, xshift=-8pt] {depth 2};
\draw[decorate,decoration={brace,amplitude=6pt}]
    (-5.2,-3.3) -- (-5.2,-3.9) node[midway, left, xshift=-8pt] {depth $L$ (leaves)};

\end{tikzpicture}

\caption{An illustration of $T_{d, L}$ with $d=3$, $L=4$.}
\end{center}
\end{figure}
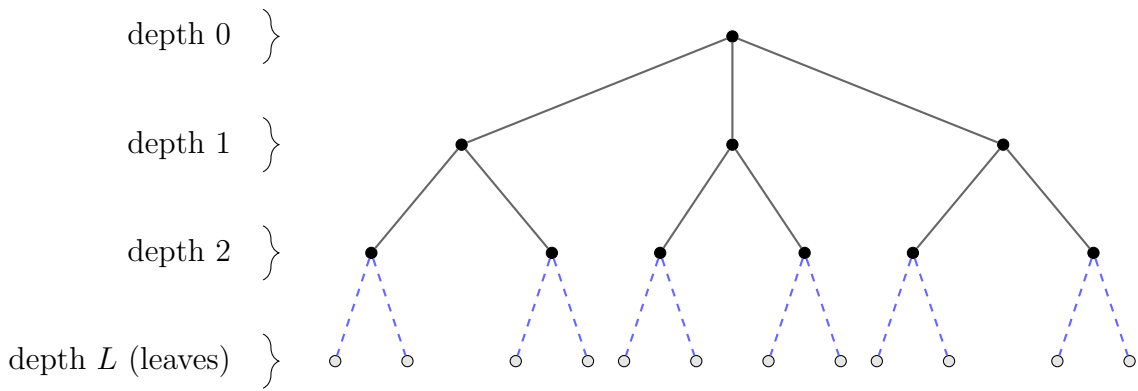

We have more refined estimates for the balls in a regular tree.
\begin{proposition}\label{prop:td1}
Let $T_{d,L}$ be the induced subgraph on $B_L(o)$ in the $d$-regular tree ($d\ge 3$). 
Then $$1-\frac{3}{d}\leq \lambda_{\max}(R_{T_{d,L}})\leq  1-\frac{4}{d}+\frac{2\sqrt{d-1}}{d}.$$
\end{proposition}
\begin{proof}
    The lower bound follows from the previous corollary.

    For the upper bound estimate, note that $V_{e,e}=\frac{4}{d}-2$ for an internal edge $e$ and $V_{e,e}=\frac{2}{d}$ for a leaf edge $e$.  By Remark~\ref{rem:schr}, we consider the line graph $\hat{T}$ of $T=T_{d,L}.$ Let $O=\{o_1,o_2,\cdots, o_d\}$ where $o_i,$ $1\leq i\leq d$, are vertices in $\hat{T}$ corresponding to edges incident to the root $o.$  Moreover, let $\hat{f}$ be the Perron eigenvector to the eigenvalue $\lambda_{\max}$ of matrix $(R_T)$. By the uniqueness of the Perron vector, $\hat{f}$ is radially symmetric since $B_L(o)$ is radially symmetric.  That is, $\hat{f}(e)=f(k)$ for any vertex $e$ in $\hat{T}$ with $k = d_c^{\hat{T}}(e,O)$, where $f:\{0,1,\cdots, L-1\}\to \mathbb{R}$, and $d_c^{\hat{T}}(\cdot,O)$ is the combinatorial distance to the set $O.$  
    Hence, the eigen-equation $(\Delta-V)\hat{f}=\lambda_{\max} \hat{f}$ on $\hat{T}$ is equivalent to the eigenvalue problem on the weighted path graph $P_{L}$ with potential.  

The depth $k = d_c^{\hat{T}}(e, O)$ takes values:
 $k = 0$: edges incident to the root (there are $d$ such edges); $1 \le k \le L-2$: internal edges at depth $k$; $k = L-1$: leaf edges.

Let $x_k = (d-1)^{k/2} f(k)$ for $0 \le k \le L-1$. Then the eigen-equation transforms into $A \mathbf{x} = -\lambda_{\max} \mathbf{x}$, where $\mathbf{x} = (x_0, x_1, \dots, x_{L-1})^\top$ and $A=(a_{ij})_{0 \le i,j \le L-1}$ is the $L \times L$ tridiagonal matrix with entries: 
$$a_{00}=\frac{3}{d}-1,a_{{L-1}{L-1}}=\frac{3}{d}, a_{ii}=\frac{4}{d}-1, 1\leq i\leq L-2, a_{k,k+1}=a_{k+1,k}=-\frac{\sqrt{d-1}}{d}, 0\leq k\leq L-2,$$ and all other entries are zero.
I.e.
\begin{align}\label{eq:matrixA}
A= \begin{pmatrix}
\frac{3}{d}-1 & -\frac{\sqrt{d-1}}{d} & 0 & \cdots & 0 \\
-\frac{\sqrt{d-1}}{d} & \frac{4}{d}-1 & -\frac{\sqrt{d-1}}{d} & \ddots & \vdots \\
0 & -\frac{\sqrt{d-1}}{d} & \ddots & \ddots & 0 \\
\vdots & \ddots & \ddots & \frac{4}{d}-1 & -\frac{\sqrt{d-1}}{d} \\
0 & \cdots & 0 & -\frac{\sqrt{d-1}}{d} & \frac{3}{d}
\end{pmatrix}
\end{align}
Since the Perron eigenvector $\hat{f}$ is positive, $\mathbf{x}$ is positive, and therefore $-\lambda_{\max}$ is the minimal eigenvalue of $A$.

By Gershgorin Disc Theorem, the eigenvalues of $A$ must lie in the union of intervals $I_i = [a_{ii} - \sum_{j\neq i} |a_{ij}|, a_{ii} + \sum_{j\neq i} |a_{ij}|]$. For $1 \leq i \leq L-2$, the interval is $\left[ \frac{4-2\sqrt{d-1}}{d}-1, \frac{4+2\sqrt{d-1}}{d}-1 \right]$. For $i=0,$ the interval is $\left[ \frac{3-\sqrt{d-1}}{d}-1, \frac{3+\sqrt{d-1}}{d}-1 \right].$ For $i=L-1,$ the interval is $\left[ \frac{3-\sqrt{d-1}}{d}, \frac{3+\sqrt{d-1}}{d} \right]$.

Since for $d \ge 3$, this yields that
$$\lambda_{\min}(A)\geq \frac{4-2\sqrt{d-1}}{d}-1.$$ This proves the result.
   
\end{proof}
Next we derive a lower bound of $\lambda_{\max}(R_{T_{d,L}})$ using a Rayleigh quotient argument.

\begin{proposition}\label{prop:td2}
Let $T_{d,L}$ be the induced subgraph on $B_L(o)$ in the $d$-regular tree ($d\ge 3$). 

Then 
\[\lambda_{\max}(R_{T_{d,L}}) \ge 1 - \frac{4}{d} + \frac{2\sqrt{d-1}}{d} - O\left(\frac{1}{L}\right).
\]
\end{proposition}

\begin{proof}
Let $A$ be the $L \times L$ tridiagonal matrix defined by Equation~\ref{eq:matrixA}. We prove
\[
\lambda_{\min}(A) \le \frac{4-2\sqrt{d-1}}{d} - 1 + O\left(\frac{1}{L}\right).
\]

For any vector $x \neq 0$ we have
\[
\lambda_{\min}(A) \le \frac{x^T A x}{x^T x}.
\]
Taking $x = (1,1,\dots,1)^L$, we compute
\begin{align*}
x^T A x
&= \sum_{i=1}^L a_{ii}
+ 2\sum_{i=1}^{L-1} a_{i,i+1}\\
&= \left(\frac{4}{d}-1\right)(L-2)
+ \frac{6}{d}-1
- 2(L-1)\frac{\sqrt{d-1}}{d}.
    \end{align*}
Dividing by $x^T x = L$, we obtain
\[
\lambda_{\min}(A)
\le \frac{x^T A x}{x^T x}
= \frac{4-2\sqrt{d-1}}{d}-1 + O\!\left(\frac{1}{L}\right).
\]
Hence, 
\[
\lambda_{\max}(R_T)
\ge 1-\frac{4}{d}+\frac{2\sqrt{d-1}}{d} - O\!\left(\frac{1}{L}\right).
\]

\end{proof}

\begin{remark}\label{remarik:limitL}
By Proposition~\ref{prop:td1} and Proposition~\ref{prop:td2}, for $d$-regular tree $T_{d, L}$ with depth $L$ and $d\geq 3,$
\[
\lim_{L \to \infty} \lambda_{\max}(R_{T_{d, L}}) = 1 - \frac{4}{d} + \frac{2\sqrt{d-1}}{d}.
\]
In particular:
\begin{itemize}
    \item For $d = 3,4$, $\lambda_{\max}(R_{T_{d, L}}) < 1$ for all $L$;
    \item For $d = 5$, $\lambda_{\max}(R_{T_{d, L}}) \to 1$ as $L \to \infty$;
    \item For $d \ge 6$, $\lambda_{\max}(R_{T_{d, L}}) > 1$ for sufficiently large $L$, with a maximum of approximately $1.236$ attained at $d = 19$.
\end{itemize}
\end{remark}

\begin{remark}
  Let $\mathbb{T}_d$ be the infinite $d$-regular tree. In fact, by standard arguments in spectral theory, one can prove that 
  \[
  \sup \operatorname{Spec}(R_{\mathbb{T}_d}) = 1 - \frac{4}{d} + \frac{2\sqrt{d-1}}{d}.
  \] 
  
  This yields that
\[
\lim_{L \to \infty} \lambda_{\max}(R_{T_{d, L}}) = \sup \operatorname{Spec}(R_{\mathbb{T}_d}).
\]
    
\end{remark}

\subsection{Universal upper bound in terms of maximum degree}
\label{subsec:max-degree-bound}

In this part, we prove that $\lambda_{\max}(R_T)$ is bounded
above by a function of the maximum vertex degree $\mathcal D$ alone
(Theorem~\ref{thm:degree-bound}).

We first show that the quadratic form $\langle f,R_Tf\rangle$ decomposes as a sum over vertices.
\begin{proposition}[Edgewise reduction]\label{prop:edgewise-reduction}
Let \(T\) be a finite tree with maximum degree \(\mathcal D\ge 3\), and let
\(f\in \mathbb R^{E(T)}\) be a unit vector. For each vertex \(v\in V(T)\), write
\[
S_v:=\sum_{e\ni v}f_e,
\qquad
A_v:=\sum_{e\ni v}f_e^2,
\qquad
d_v:=\deg_T(v).
\]
Choose a root \(r\) with \(d_r\ge 3\), and orient every edge away from \(r\).
For each non-root vertex \(v\), let \(e_p^{(v)}\) denote the parent edge and let
\(E_c^{(v)}\) denote the set of child edges. Then:

\begin{enumerate}
\item The quadratic form of \(R_T\) admits the vertex decomposition
\[
\langle f,R_Tf\rangle
=
\sum_{v\in V(T)}\frac{S_v^2-2A_v}{d_v}.
\]

\item For any choice of parameters \(\alpha_v>0\) respect to vertex $v$, define
\[
A(d,\alpha):=\frac{\left(1+\frac1\alpha\right)(d-1)-2}{d},
\qquad
B(d,\alpha):=\frac{\alpha-1}{d}.
\]
Then
\[
\frac{S_v^2-2A_v}{d_v}
\le
B(d_v,\alpha_v) f_{e_p^{(v)}}^2
+
A(d_v,\alpha_v)\sum_{e\in E_c^{(v)}}f_e^2
\qquad (v\neq r),
\]
and
\[
\frac{S_r^2-2A_r}{d_r}
\le
\frac{d_r-2}{d_r}\sum_{e\in E_c^{(r)}}f_e^2.
\]
After regrouping the edge contributions, one obtains
\[
\langle f,R_Tf\rangle\le \max_{e\in \vec E} C_e,
\]
where
\[
C_e=
\begin{cases}
A(d_u,\alpha_u)+B(d_v,\alpha_v), & e=(u\to v),\ u\neq r,\\[4pt]
\dfrac{d_r-2}{d_r}+B(d_v,\alpha_v), & e=(r\to v).
\end{cases}
\]

\item For the linear choice \(\alpha_d=Kd+1\) with \(0<K<1/3\), define
\[
C(d,K):=
\frac{\left(1+\frac{1}{Kd+1}\right)(d-1)-2}{d}+K.
\]
Then every root-edge coefficient is dominated by the corresponding internal-edge model:
\[
\frac{d-2}{d}+K<C(d,K)
\qquad (d\ge 3).
\]
Consequently,
\[
\langle f,R_Tf\rangle\le \max_{1\le d\le \mathcal D,\ d\in\mathbb Z} C(d,K).
\]
\end{enumerate}
\end{proposition}

\begin{proof}
The vertex decomposition follows by collecting the diagonal and off-diagonal
contributions of \(R_T\) at each vertex:
\[
\langle f,R_Tf\rangle
=
\sum_{v\in V(T)}
\frac{1}{d_v}
\left(
2\sum_{\substack{e<e'\\ e,e'\ni v}}f_ef_{e'}
-2\sum_{e\ni v}f_e^2
\right)
=
\sum_{v\in V(T)}\frac{S_v^2-2A_v}{d_v}.
\]
For \(v\neq r\), the weighted Cauchy inequality gives
\[
S_v^2
\le
(1+\alpha_v)f_{e_p^{(v)}}^2
+
\left(1+\frac1{\alpha_v}\right)(d_v-1)\sum_{e\in E_c^{(v)}}f_e^2,
\]
and substituting this into \((S_v^2-2A_v)/d_v\) yields the stated bound with
coefficients \(A(d_v,\alpha_v)\) and \(B(d_v,\alpha_v)\). At the root,
\(S_r^2\le d_rA_r\) gives the root estimate. Regrouping the coefficients
edge by edge proves the edgewise bound.

For the linear choice \(\alpha_d=Kd+1\),
\[
B(d,Kd+1)=K,
\]
and
\[
A(d,Kd+1)-\frac{d-2}{d}
=
\frac{(1-K)d-2}{d(Kd+1)}.
\]
If \(d\ge 3\) and \(0<K<1/3\), then the denominator is positive and
\[
(1-K)d-2\ge 3(1-K)-2=1-3K>0.
\]
Hence
\[
\frac{d-2}{d}+K<C(d,K)
\qquad (d\ge 3),
\]
so every root-edge coefficient is strictly smaller than the internal-edge
model with the same initial degree. The final estimate follows.
\end{proof}

Next, we need the following lemma. 
\begin{lemma}
\label{lem:coefficient-optimization}
Let
\[
C(d,K):=
\frac{\left(1+\frac{1}{Kd+1}\right)(d-1)-2}{d}+K,
\qquad d\ge 1.
\]
For an integer \(m\) with \(3\le m\le 19\), set
\[
K_m:=\frac{\sqrt{m-1}-1}{m}.
\]
Then the following hold.

\begin{enumerate}
\item If \(3\le m\le 18\), then \(C(d,K_m)\) is strictly increasing on
\([1,m]\). Hence
\[
C(d,K_m)\le C(m,K_m)
\qquad (1\le d\le m).
\]

\item If \(m=19\), then
\[
C(d,K_{19})\le C(19,K_{19})
\qquad (d\in\mathbb Z_{\ge 1}).
\]
\end{enumerate}

Moreover, for every \(3\le m\le 19\),
\[
C(m,K_m)
=
1-\frac{4}{m}
+
\frac{2\sqrt{m-1}}{m}.
\]
\end{lemma}

\begin{proof}
First note that
\[
K_m \times m+1=\sqrt{m-1}.
\]
Therefore
\[
C(m,K_m)
=
\frac{
\left(1+\frac{1}{\sqrt{m-1}}\right)(m-1)-2
}{m}
+
\frac{\sqrt{m-1}-1}{m}.
\]
Simplifying gives
\[
C(m,K_m)
=
1-\frac{4}{m}
+
\frac{2\sqrt{m-1}}{m}.
\]

Next, for fixed \(K\), direct differentiation gives
\[
\frac{\partial C}{\partial d}
=
\frac{
K(3K-1)d^2+8Kd+4
}{
d^2(Kd+1)^2
}.
\tag{1}
\]
For \(3\le m\le 19\), one has
\[
0<K_m<\frac13.
\]
Indeed, \(K_m>0\) is clear, and \(K_m<1/3\) is equivalent to
\[
3\sqrt{m-1}<m+3.
\]
Writing \(t=\sqrt{m-1}\), this becomes
\[
3t<t^2+4,
\]
which follows from \(t^2-3t+4>0\).

Now assume \(3\le m\le 18\). Let
\[
N_m(d):=K_m(3K_m-1)d^2+8K_md+4.
\]
Since \(0<K_m<1/3\), \(N_m\) is a concave quadratic polynomial, and
\(N_m(0)=4>0\). Put \(t=\sqrt{m-1}\). Then
\[
m=t^2+1,
\qquad
K_m=\frac{t-1}{t^2+1}.
\]
A direct substitution gives
\[
N_m(m)=t(-t^2+4t+1).
\]
Since \(m\le 18\), we have
\[
t=\sqrt{m-1}\le \sqrt{17}<2+\sqrt5.
\]
Thus
\[
-t^2+4t+1>0,
\]
so \(N_m(m)>0\). By concavity and \(N_m(0)>0\), it follows that
\(N_m(d)>0\) for all \(d\in[1,m]\). Hence, by (1),
\[
\frac{\partial C}{\partial d}>0
\qquad (1\le d\le m),
\]
and \(C(d,K_m)\) is strictly increasing on \([1,m]\).

It remains to consider \(m=19\). For \(K_{19}=(3\sqrt2-1)/19\), the
numerator \(N_{19}\) in (1) is concave and has its unique positive zero
in \((18,19)\). Indeed,
\[
N_{19}(18)>0,
\qquad
N_{19}(19)<0.
\]
Hence \(C(d,K_{19})\) is increasing before this critical point and
decreasing after it. Therefore the maximum over integer \(d\ge1\) is
attained at either \(d=18\) or \(d=19\). Finally,
\[
C(19,K_{19})>C(18,K_{19}).
\]
Consequently,
\[
C(d,K_{19})\le C(19,K_{19})
\qquad (d\in\mathbb Z_{\ge1}).
\]
The proof is complete.
\end{proof}
\begin{comment}
\begin{theorem}\label{thm:degree-bound}
Let $T=(V,E)$ be a finite discrete Einstein tree, let
\[
\mathcal D:=\max_{v\in V} d_v.
\]
 Then the maximal eigenvalue satisfies
\[
\lambda_{\max}(R_T)\le
\begin{cases}
-2, & \mathcal D=1,\\[4pt]
1-\dfrac{4}{\mathcal D}+\dfrac{2\sqrt{\mathcal D-1}}{\mathcal D}, & 2\le \mathcal D\le 18,\\[8pt]
\dfrac{15+6\sqrt2}{19}, & \mathcal D\ge 19.
\end{cases}
\]
\end{theorem}

\end{comment}

\begin{proof}[Proof of Theorem~\ref{thm:degree-bound}]
We first handle the degenerate cases. If $\mathcal D=1$, then $T$ is a single edge with $R_T=(-2)$,
so $\lambda_{\max}=-2$.
If $\mathcal D=2$, then $T$ is a path; by
Example~\ref{prop:path-eigenpair}, $\lambda_{\max}(R_T)\le 0$,
and the formula gives $1-4/2+2\sqrt1/2=0$.

Now assume \(\mathcal D\ge 3\). Let \(f\) be any unit vector in
\(\mathbb R^{E(T)}\). By Proposition~\ref{prop:edgewise-reduction}, for
every \(K\in(0,1/3)\),
\[
\langle f,R_Tf\rangle\le \max_{1\le d\le \mathcal D,\ d\in\mathbb Z} C(d,K),
\]
where
\[
C(d,K)=\frac{\left(1+\frac{1}{Kd+1}\right)(d-1)-2}{d}+K.
\]
Set
\[
m:=\min\{\mathcal D,19\},
\qquad
K:=K_m=\frac{\sqrt{m-1}-1}{m}.
\]
By Lemma~\ref{lem:coefficient-optimization}, every integer degree
\(d\le \mathcal D\) satisfies
\[
C(d,K)\le C(m,K_m)
=
1-\frac{4}{m}
+
\frac{2\sqrt{m-1}}{m}.
\]
Indeed, if \(3\le \mathcal D\le 18\), then \(m=\mathcal D\), and the
lemma applies on \([1,\mathcal D]\). If \(\mathcal D\ge 19\), then
\(m=19\), and the lemma gives the bound for all integers \(d\ge 1\).

Therefore
\[
\langle f,R_Tf\rangle
\le C(m,K_m).
\]
Since \(f\) was arbitrary, the Rayleigh--Ritz principle yields
\[
\lambda_{\max}(R_T)\le C(m,K_m).
\]

If \(3\le \mathcal D\le 18\), then \(m=\mathcal D\), and hence
\[
\lambda_{\max}(R_T)
\le
1-\frac{4}{\mathcal D}
+
\frac{2\sqrt{\mathcal D-1}}{\mathcal D}.
\]
If \(\mathcal D\ge 19\), then \(m=19\), and hence
\[
\lambda_{\max}(R_T)
\le
1-\frac{4}{19}
+
\frac{2\sqrt{18}}{19}
=
\frac{15+6\sqrt2}{19}.
\]

\end{proof}
\begin{remark}
Let
\[
F(x):=1-\frac{4}{x}+\frac{2\sqrt{x-1}}{x}\qquad (x\ge 2).
\]
If $x$ is viewed as a continuous variable, then $F$ attains its maximum at
\[
x_\ast=10+4\sqrt5\approx 18.944,
\]
and
\[
F(x_\ast)=\sqrt5-1.
\]
However, the maximum degree $\mathcal D$ of a tree is an integer. Hence the discrete maximum becomes
\[
\max_{\mathcal D\in\mathbb Z_{\ge 2}}F(\mathcal D)=F(19)=\frac{15+6\sqrt2}{19}<\sqrt5-1.
\]
\end{remark}

\begin{proposition}
Let $T$ be a finite discrete Einstein tree with maximum degree $\mathcal D.$
Then
\[
\lambda_{\max}(R_T)\ge -\frac{2}{\mathcal D}
=\lambda_{\max}(R_{S_{\mathcal D}}).
\]
\end{proposition}

\begin{proof}
Choose a vertex $v$ of degree $\mathcal D$, and let $E_v=\{e_1,\dots,e_{\mathcal D}\}$ be the set of incident edges. Define a test vector $g\in\mathbb R^{|E|}$ by
\[
g_e=
\begin{cases}
\dfrac{1}{\sqrt{\mathcal D}},& e\in E_v,\\[4pt]
0,& e\notin E_v.
\end{cases}
\]
Then $\|g\|_2=1$. Write the neighbors of $v$ as $u_1,\dots,u_{\mathcal D}$, where $e_i=\{v,u_i\}$. Using the vertex form of the quadratic form,
\[
\langle g,R_Tg\rangle=\sum_{x\in V}\frac{S_x(g)^2-2A_x(g)}{d_x},
\]
the only nonzero contributions come from $v$ and its neighbors $u_1,\dots,u_{\mathcal D}$. They are
\[
\frac{S_v(g)^2-2A_v(g)}{d_v}
=\frac{\mathcal D-2}{\mathcal D},
\qquad
\frac{S_{u_i}(g)^2-2A_{u_i}(g)}{d_{u_i}}
=-\frac{1}{\mathcal D\, d_{u_i}}.
\]
Thus
\[
\langle g,R_Tg\rangle
=\frac{\mathcal D-2}{\mathcal D}-\frac{1}{\mathcal D}\sum_{i=1}^{\mathcal D}\frac{1}{d_{u_i}}
\ge \frac{\mathcal D-2}{\mathcal D}-1
=-\frac{2}{\mathcal D},
\]
since $d_{u_i}\ge 1$ for all $i$. 
By the Rayleigh--Ritz principle,
\[
\lambda_{\max}(R_T)\ge \langle g,R_Tg\rangle\ge -\frac{2}{\mathcal D}.
\]
\qedhere
\end{proof}

\begin{corollary}
Let $T$ be a finite discrete Einstein tree with maximum degree $\mathcal D$. Then
\[
-\frac{2}{\mathcal D}\le \lambda_{\max}(R_T)\le
\begin{cases}
-2, & \mathcal D = 1,\\[4pt]
1-\dfrac{4}{\mathcal D}+\dfrac{2\sqrt{\mathcal D-1}}{\mathcal D}, & 2\le \mathcal D\le 18,\\[8pt]
\dfrac{15+6\sqrt2}{19}, & \mathcal D\ge 19.
\end{cases}
\]
In particular, for fixed maximum degree $\mathcal D$, the maximal eigenvalue is trapped between the star benchmark and the degree-dependent upper bound.
\end{corollary}

\section{Monotonicity of the Perron Eigenvalue and a Topological Classification of Trees}\label{sec:eigenvalueinc}

Next, we study how $\lambda_{\max}(R_T)$ changes under attachment.
Proposition~\ref{pro:attachment} gives a single bound $\Lambda(d,k)$ on $\lambda_{\max}(R_T)$
below which attaching a forest of $k$ trees at a degree-$d$ vertex cannot decrease the
eigenvalue; in particular it never decreases for $d\le 2$, and always strictly increases
when $\lambda_{\max}(R_T)\le 0$.

\begin{proposition}\label{pro:attachment}
Let $T$ be a tree, $v\in V(T)$ a vertex of degree $d$, and let $T'$ be obtained from $T$
by  attaching a forest of $k\ge 1$ trees to $v$ via $k$ new edges $e_1,\dots,e_k$,
where $e_j$ joins $v$ to a vertex $u_j$ of the forest. Set
$\lambda:=\lambda_{\max}(R_T)$ and
\[
\Lambda(d,k):=\begin{cases}+\infty,& d\le 2,\\[3pt]\dfrac{4}{(d+k)(d-2)},& d\ge 3.\end{cases}
\]
If $\lambda\le\Lambda(d,k)$, then $\lambda_{\max}(R_{T'})\ge\lambda$, with strict
inequality whenever $\lambda<\Lambda(d,k)$. In particular:
\begin{enumerate}
\item[(a)] if $d\le 2$, the conclusion holds for every $\lambda$ and every $k$;
\item[(b)] if $\lambda\le 0$, the eigenvalue strictly increases, for every $d$ and $k$;
\item[(c)] if $d=3$, it holds for all $0\le\lambda\le\frac{4}{3+k}$, in particular for $\lambda=0$.
\end{enumerate}
\end{proposition}

\begin{proof}
Let $w>0$ be the Perron eigenvector, $\|w\|=1$, $R_Tw=\lambda w$. Let $f_1,\dots,f_d$ be
the edges incident to $v$ and put $S=\sum_{i=1}^d w_{f_i}>0$, $A=\sum_{i=1}^d w_{f_i}^2>0$,
so $0<S^2\le dA$ by Cauchy--Schwarz.  Extend $w$ to $\tilde w$ on $E(T')$ by giving weight
$y>0$ to each of the $k$ new edges, $w_f$ to each $f\in E(T)$, and $0$ to every edge inside
the attached forest; then $\|\tilde w\|^2=1+ky^2$. 

In the vertex decomposition $\langle g,R_{T} g\rangle=\sum_z(S_z^2-2A_z)/d_z$
(cf.\ Proposition~\ref{prop:edgewise-reduction}), passing from $T$ to $T'$ changes only the
term at $v$ (degree $d\to d+k$) and adds one term per new endpoint $u_j$. As $u_j$ has degree $d_{u_j}\ge 1$ and
carries only the edge weight $y$, its term is $-y^2/d_{u_j}\ge -y^2$, with equality iff
$d_{u_j}=1$. Hence $\tilde w^\top R_{T'}\tilde w$
is smallest when every $u_j$ is a leaf, and it suffices to treat that case, where
\[
\tilde w^\top R_{T'}\tilde w=\lambda+k\!\left[\frac{2A-S^2}{d(d+k)}+\frac{2S}{d+k}\,y
-\frac{d+2}{d+k}\,y^2\right].
\]
By the variational principle $\lambda_{\max}(R_{T'})\ge\tilde w^\top R_{T'}\tilde w/(1+ky^2)$,
and this is $\ge\lambda$ (resp.\ $>\lambda$) as soon as
\[
g(y):=\frac{2A-S^2}{d(d+k)}+\frac{2S}{d+k}\,y-C_k\,y^2\ \ge\ 0\ (\text{resp.\ }>0)
\quad\text{for some }y>0,\qquad C_k:=\frac{d+2}{d+k}+\lambda.
\]
If $C_k\le0$, then $-C_k\ge0$ and the linear coefficient $\frac{2S}{d+k}>0$, so $g(y)>0$
for large $y$. If $C_k>0$, $g$ is concave with
\[
\max_y g=\frac{2A-S^2}{d(d+k)}+\frac{S^2}{(d+k)^2C_k}
=\frac{2(d+k)C_kA+S^2\bigl(d-(d+k)C_k\bigr)}{d(d+k)^2C_k}.
\]
If $d-(d+k)C_k\ge0$ the numerator is $\ge 2(d+k)C_kA>0$; otherwise it is decreasing in
$S^2$, so by $S^2\le dA$ it is at least $A\bigl[d^2-(d-2)(d+k)C_k\bigr]$. For $d\le2$ this
bracket is $\ge d^2>0$; for $d\ge3$ it is $\ge0$ exactly when
$C_k\le\frac{d^2}{(d-2)(d+k)}$, i.e.\ $\lambda\le\frac{4}{(d+k)(d-2)}=\Lambda(d,k)$, strictly
when $\lambda<\Lambda(d,k)$. Hence $\max_y g\ge0$, strictly for $\lambda<\Lambda(d,k)$.
\end{proof}

\begin{remark}
For the adjacency or Laplacian matrix, attaching a leaf never decreases the spectral
radius. The Ricci matrix $R_T$ behaves differently: Proposition~\ref{pro:attachment}
guarantees that $\lambda_{\max}$ does not decrease as long as
$\lambda_{\max}(R_T)\le\Lambda(d,k)$, and by part~(b) it strictly increases whenever
$\lambda_{\max}(R_T)\le 0$, irrespective of the attachment vertex. Beyond this regime the
eigenvalue can genuinely drop: the appendix exhibits attachments at vertices of degree
$\ge 4$ (necessarily with $\lambda_{\max}(R_T)>0$) for which $\lambda_{\max}$ strictly
decreases.
\end{remark}

\begin{definition}[$S_3^2$ tree]
\label{def:S32}
The tree $S_3^2$ is the tree obtained from a star $S_3$,  a central vertex with three leaves,  by subdividing each edge once. Equivalently, $S_3^2$ consists of a central vertex $c$ connected to three vertices $a_1, a_2, a_3$, and each $a_i$ is connected to a leaf $b_i$. The tree has $7$ vertices and $6$ edges. 

$S_3^2$ is the minimal tree (with respect to edge deletion) such that deleting any edge yields a caterpillar tree.  By direct computation, $\lambda(S_3^2)=0$. 
\end{definition}

\begin{figure}[h]
\centering
\begin{tikzpicture}[
every node/.style={
    draw,
    circle,
    fill=black,
    inner sep=1.5pt
},
thick]

\node (c)  at (0,0)     [label=below:$c$] {};
\node (a1) at (-1.5,1.2)[label=left:$a_1$] {};
\node (a2) at (0,1.5)   [label=above:$a_2$] {};
\node (a3) at (1.5,1.2) [label=right:$a_3$] {};

\node (b1) at (-2.5,2)  [label=left:$b_1$] {};
\node (b2) at (0,2.5)   [label=above:$b_2$] {};
\node (b3) at (2.5,2)   [label=right:$b_3$] {};

\draw (c) -- (a1);
\draw (c) -- (a2);
\draw (c) -- (a3);
\draw (a1) -- (b1);
\draw (a2) -- (b2);
\draw (a3) -- (b3);

\end{tikzpicture}

\caption{The tree $S_3^2$: a central vertex $c$ connected to three paths of length $2$.}
\label{fig:S32}
\end{figure}
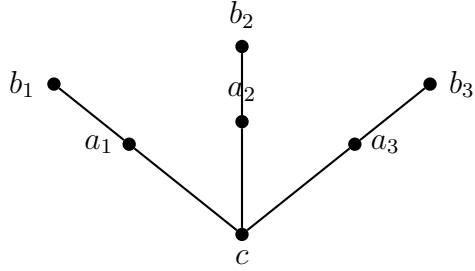

Now we are ready to prove the following main result.
\begin{proof}[Proof of Theorem~\ref{the:lambda-negative-caterpillar}]
Since $T$ is not a caterpillar, it contains a copy of $S_3^2$; fix one and let $c$ be its
center. We build $T$ up from this $S_3^2$, starting at $c$. While the tree is still the
bare $S_3^2$, attach to $c$ all subtrees of $T$ incident to $c$ apart from the three legs
of $S_3^2$; since $d_c=3$ and the current eigenvalue is $\lambda_{\max}(R_{S_3^2})=0$,
Proposition~\ref{pro:attachment}(c) shows that the resulting tree still satisfies
$\lambda_{\max}\ge 0$. The rest of $T$ is then assembled bottom-up, each subtree being
attached to its parent only once it is complete; in this order every attachment is made at
a vertex that still has degree at most $2$, and hence does not decrease $\lambda_{\max}$ by
Proposition~\ref{pro:attachment}(a). As $c$ is the only vertex ever treated at degree $3$,
the eigenvalue never decreases along the way, and therefore
$\lambda_T\ge\lambda_{\max}(R_{S_3^2})=0$, with equality only for $T=S_3^2$.
\end{proof}

\section{Monotone Decay of Edge Weights for $\lambda<0$}
\label{perro<0}

In this section we describe the structure of the Perron weights---equivalently, the
Einstein metric---in the positive-curvature regime $\lambda_{\max}(R_T)<0$
(i.e.\ $\kappa>0$).

By Theorem~\ref{theo:directional-decreasing}, the Perron weights strictly decrease along
any path running outward from a maximal-weight edge toward the leaves.
Theorem~\ref{coro:uniquemax} further restricts the maxima: a non-star tree has at most
two edges of maximal weight, and when there are two they share a vertex of degree two.
Both are proved directly, without using the caterpillar classification.

\begin{theorem}[Directional Decreasing Chain]
\label{theo:directional-decreasing}
Let $T$ be a tree with Perron eigenpair $(\lambda, w)$ where $\lambda < 0$. 
Suppose there exists a directed edge $u \xrightarrow{e} v$ such that $S_u < 2w_e$. 
Then for any neighbor $x$ of $u$ with $x \neq v$, the edge $f = \{u, x\}$ satisfies $w_f < w_e$. 

Moreover, along any path starting from $u$ in the direction away from $v$, the edge weights 
form a strictly decreasing sequence
\[
w_e > w_f > w_g > \cdots > w_{\ell}
\]
until reaching a leaf $\ell$, where $w_{\ell}$ is the weight of the leaf edge.
\end{theorem}

\begin{proof}
We first show that $w_f < w_e$ for any neighbor $x \neq v$. 
From $S_u < 2w_e$ and $S_u = w_e + \sum_{y \sim u, y \neq v} w_{uy}$, we have
\[
\sum_{y \sim u, y \neq v} w_{uy} < w_e.
\]
Hence $w_f = w_{ux} \le \sum_{y \sim u, y \neq v} w_{uy} < w_e$.

Next, we prove the inductive step. Since $w_e > w_f$ and $S_u \ge w_e + w_f$, we obtain
$S_u > 2w_f$. The eigenvalue equation for $f = \{u, x\}$ gives
\[
\frac{S_u - 2w_f}{d_u} + \frac{S_x - 2w_f}{d_x} = \lambda w_f < 0.
\]
The first term is positive, so the second must be negative, yielding $S_x < 2w_f$. 
Consequently, for any neighbor $z$ of $x$ with $z \neq u$, the edge $g = \{x, z\}$ satisfies
$w_g < w_f$ (by the same argument as the first step). Moreover, $S_x > 2w_g$, so the condition 
$S_x < 2w_f$ propagates to the next vertex.

By induction, the edge weights along any path starting from $u$ away from $v$ strictly decrease. 
Since the tree is finite, the sequence must terminate at a leaf.
\end{proof}

\begin{corollary}[Leaf Edge Minimality for $\lambda < 0$]
\label{cor:leaf-minimality}
Let $T$ be a tree with Perron eigenpair $(\lambda, w)$ where $\lambda < 0$.
The minimum weight is only attained in the leaf edges, i.e. there is a leaf edge $e_0$ such that 
$$w_{e_0}<w_f, \quad \forall\ \mathrm{internal\ edge}\ f.$$

\end{corollary}

\begin{proof}
Starting from any edge $e$, repeated application of Theorem~\ref{theo:directional-decreasing} along a maximal directed path produces a strictly decreasing sequence of edge weights terminating at a leaf edge. Hence every edge weight is bounded below by a leaf-edge weight, which shows that leaf edges attain the global minimum.
\end{proof}

We are ready to prove one of main results.
\begin{proof}[Proof of Theorem~\ref{coro:uniquemax}]
Let $e_0$ be a leaf edge minimizing $w_e$. Moving inward along the unique path away from the leaf, weights strictly increase by Theorem~\ref{theo:directional-decreasing}. Let $e_{\max} = \{x,y\}$ be an edge of maximal weight encountered along this path, with $x$ closer to the leaf. By construction, $S_x < 2w_{e_{\max}}$.

If $S_y < 2w_{e_{\max}}$, we are done. If $S_y = 2w_{e_{\max}}$, then $d_y = 2$ and the other edge $e' = \{y,z\}$ incident to $y$ must also have weight $w_{e_{\max}}$, and $S_z < 2w_{e'}$ (as $\lambda<0$). Applying Theorem~\ref{theo:directional-decreasing} to the edge $e'$, we obtain a strictly decreasing chain beyond $z$, so $e_{\max}$ and $e'$ are the only edges achieving the maximum weight, and they share the vertex $y$ of degree $2$.
\end{proof}

\begin{remark}
When $\lambda = 0$, the strict inequalities in Theorem~\ref{theo:directional-decreasing} 
degenerate into equalities. Consequently, edge weights may remain constant along degree-$2$ 
chains, as exemplified by the subdivided double-star trees $D_{3,3}^{(k)}$ (see 
Example~\ref{ex:double-star} and Proposition~\ref{pro:D33k}), where all internal edges carry the same weight. 
Thus $\lambda = 0$ marks the transition between strict monotone decay ($\lambda < 0$) 
and flat, subdivision-invariant configurations ($\lambda = 0$).
\end{remark}

\section{Leaf Weights and Extremal Edges}\label{sec:leafminmax}

We describe the Perron vector near leaves. Leaf edges at a common vertex have equal
weight, strictly below any incident internal edge (Corollary~\ref{cor:leaf-local}). The
global maximum always lies on an internal edge
(Proposition~\ref{pro:schrodingemaximalweight}); the global minimum lies on a leaf edge
for $\lambda\le 0$, but can lie on an internal edge for $\lambda>0$
(Proposition~\ref{pro:leaf-edge-min}).

\begin{corollary}[Local leaf monotonicity]
\label{cor:leaf-local}
Let $(\lambda, w)$ be the Perron eigenpair of $R_T$.

Then for any vertex $y$:
\begin{enumerate}
    \item All leaf edges incident to $y$ have equal weight.
    \item Any leaf edge incident to $y$ has strictly smaller weight than any internal edge incident to $y$.
\end{enumerate}
\end{corollary}

\begin{proof}
(1) Let $e_1,\dots,e_k$ be leaf edges incident to $y$. For each $e_i = \{y, z_i\}$, the eigenvalue equation with $d_{z_i}=1$ gives$
-1 + \frac{S_y - 2w_{e_i}}{w_{e_i}d_y} = \lambda.$
Thus $\frac{S_y}{w_{e_i}} = d_y(\lambda+1) + 2$ is constant across all $i$, so $w_{e_1} = \cdots = w_{e_k}$.

(2) Let $e$ be a leaf edge and $f = \{y,u\}$ an internal edge incident to $y$, with $d_u \ge 2$. Their eigenvalue equations are
\begin{align*}
-1 + \frac{S_y - 2w_e}{w_e d_y} &= \lambda, \\
\frac{S_y - 2w_f}{w_f d_y} + \frac{S_u - 2w_f}{w_f d_u} &= \lambda. 
\end{align*}
Suppose $w_e \ge w_f$. Then $\frac{S_y - 2w_e}{w_e} \le \frac{S_y - 2w_f}{w_f}$, and dividing by $d_y$ yields
\[
1+\lambda = \frac{S_y - 2w_e}{w_e d_y} \le \frac{S_y - 2w_f}{w_f d_y} = \lambda - \frac{S_u - 2w_f}{w_f d_u}.
\]
Thus $1 \le -\frac{S_u - 2w_f}{w_f d_u}$, implying $S_u \le (2-d_u)w_f \le 0$, a contradiction. Hence $w_e < w_f$.
\end{proof}

\begin{proposition}\label{pro:schrodingemaximalweight}
Let $T$ be a tree that is not a star, with Perron eigenpair $(\lambda,w)$. Then the maximal
weight is attained only on internal edges.
\end{proposition}

\begin{proof}
We only need to see the case  $\lambda > 0$. 
The eigenvalue equation at the maximum edge $e^* = \{u,v\}$ can be rearranged as:
\begin{equation*}
\lambda = \frac{1}{d_u} \left( \frac{S_u}{w_{e^*}} - 2 \right) + \frac{1}{d_v} \left( \frac{S_v}{w_{e^*}} - 2 \right)
\end{equation*}
Since $w_{e^*} \geq w_f$ for all $f$, it follows that for any vertex $p$ with $d_p=2$, the term $(S_p/w_{e^*} - 2) \leq 0$. Consequently, for the total sum $\lambda$ to be non-negative, at least one endpoint (say $v$) must satisfy $S_v > 2w_{e^*}$, which forces $d_v\ge 3$. Together with Corollary~\ref{cor:leaf-local}, $e^*$ must be an internal edge.  
\end{proof}

\begin{proposition}
\label{pro:leaf-edge-min}
 Let $T$ be a tree with Perron eigenpair $(\lambda, w)$. If $\lambda \le 0$, the minimal weight is attained on a leaf edge.  If $\lambda > 0$, the minimal weight may attained on internal edge.  
\end{proposition}

\begin{proof}
Let $e = \{u,v\}$ be a minimal-weight edge, and suppose for contradiction that $d_u, d_v \ge 2$.  
Since $w_e \le w_f$ for all $f$, we have $S_u \ge d_u w_e$ and $S_v \ge d_v w_e$. Hence
\[
S_u - 2w_e \ge (d_u-2)w_e \ge 0,\qquad 
S_v - 2w_e \ge (d_v-2)w_e \ge 0.
\]
The eigenvalue equation for $e$ with $\lambda \le 0$ gives
\[
\frac{S_u - 2w_e}{d_u} + \frac{S_v - 2w_e}{d_v} = 0.
\]
This forces $\lambda=0$, $d_u=d_v=2$, and the two edges adjacent to $e$ to also have weight
$w_e$. The same argument propagates the equality and the degree-$2$ condition outward; since
$T$ is finite, the chain must reach a leaf edge, which by Corollary~\ref{cor:leaf-local} is
strictly smaller than the adjacent internal edge of weight $w_e$,  contradicting the
minimality of $e$. Hence at least one endpoint of $e$ has degree $1$, i.e.\ $e$ is a leaf edge.

For $\lambda > 0$, see Example~\ref{ex:D4k29} which exhibits counterexamples where the minimum occurs on an internal edge.
\end{proof}

\section{Examples and spectral phase transition}
\label{sec:examples}
In this section, we present explicit families of trees that illustrate the spectral phase transition of $\lambda_{\max}$ from negative to zero to positive. We also exhibit a pair of non-isomorphic trees that are cospectral with respect to $R_T$, showing that the full spectrum does not uniquely determine the tree structure---in contrast to the Perron eigenpair which captures finer geometric information.

\subsection{Perron Structure vs. Full Spectrum}
\label{subsec:perron-vs-spectrum}
The full spectrum of $R_T$ does not always distinguish non-isomorphic trees. The following example gives two trees on $17$ vertices that are cospectral but have distinct Perron eigenvectors.

\begin{example}[A Ricci-flow cospectral pair]\label{ex:cospectral-pair}
Consider the following two trees $T_1$ and $T_2$ with 17 vertices each, defined by their edge sets:

\begin{gather*}
E(T_1) = \{ (1, 0), (1, 2), (1, 6), (1, 8), (0, 9), (0, 15), (0, 16), \\
          (2, 3), (3, 4), (4, 5), (6, 7), (9, 10), (10, 11), \\
          (10, 13), (10, 14), (11, 12) \} \\[4pt]
E(T_2) = \{ (1, 0), (1, 2), (1, 6), (1, 7), (0, 8), (2, 3), (3, 4), \\
          (4, 5), (8, 9), (8, 14), (8, 16), (9, 10), (9, 12), \\
          (9, 13), (10, 11), (14, 15) \}
\end{gather*}
This is the smallest trees where the spectrum of $R_T$ matrix fails to distinguish the global structure.

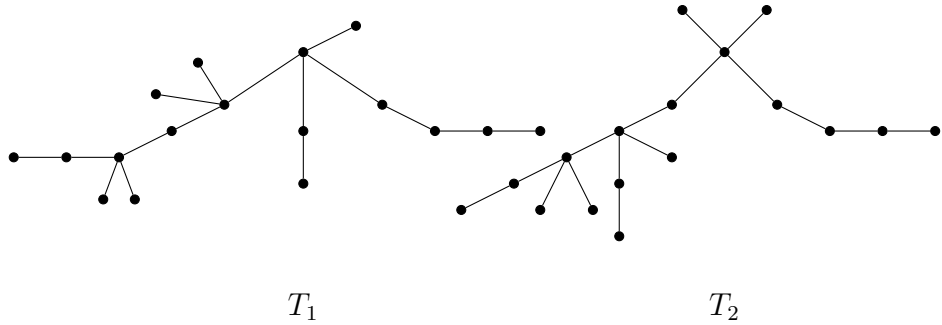
\begin{figure}[H]
    \centering
    \begin{tikzpicture}[scale=0.7, every node/.style={draw, circle, inner sep=1.2pt, fill=black}, edge from parent/.style={draw, -}]
        
        % Tree T1
        \begin{scope}[shift={(0,0)}]
            \node (v1) at (0,0) {} ; % node 1
            \node (v0) at (-1.5, -1) {}; % node 0
            \node (v2) at (1.5, -1) {}; % node 2
            \node (v6) at (0, -1.5) {}; % node 6
            \node (v8) at (1, 0.5) {}; % node 8
            
            % Neighbors of 0
            \node (v9) at (-2.5, -1.5) {};
            \node (v15) at (-2, -0.2) {};
            \node (v16) at (-2.8, -0.8) {};
            
            % Neighbors of 2, 3, 4, 5
            \node (v3) at (2.5, -1.5) {}; \node (v4) at (3.5, -1.5) {}; \node (v5) at (4.5, -1.5) {};
            
            % Neighbors of 6, 7
            \node (v7) at (0, -2.5) {};
            
            % Neighbors of 9, 10, 11, 12, 13, 14
            \node (v10) at (-3.5, -2) {};
            \node (v11) at (-4.5, -2) {}; \node (v12) at (-5.5, -2) {};
            \node (v13) at (-3.8, -2.8) {}; \node (v14) at (-3.2, -2.8) {};

            \draw (v1)--(v0) (v1)--(v2) (v1)--(v6) (v1)--(v8);
            \draw (v0)--(v9) (v0)--(v15) (v0)--(v16);
            \draw (v2)--(v3) (v3)--(v4) (v4)--(v5);
            \draw (v6)--(v7);
            \draw (v9)--(v10) (v10)--(v11) (v11)--(v12) (v10)--(v13) (v10)--(v14);
            
            \node[draw=none, fill=none, below=2cm] at (0,-1.5) {$T_1$};
        \end{scope}

        % Tree T2
        \begin{scope}[shift={(8,0)}]
            \node (u1) at (0,0) {}; % node 1
            \node (u0) at (-1, -1) {}; 
            \node (u2) at (1, -1) {};
            \node (u6) at (0.8, 0.8) {};
            \node (u7) at (-0.8, 0.8) {};
            
            \node (u8) at (-2, -1.5) {};
            \node (u3) at (2, -1.5) {}; \node (u4) at (3, -1.5) {}; \node (u5) at (4, -1.5) {};
            
            \node (u9) at (-3, -2) {};
            \node (u14) at (-2, -2.5) {};
            \node (u16) at (-1, -2) {};
            
            \node (u10) at (-4, -2.5) {};
            \node (u12) at (-3.5, -3) {};
            \node (u13) at (-2.5, -3) {};
            
            \node (u11) at (-5, -3) {};
            \node (u15) at (-2, -3.5) {};

            \draw (u1)--(u0) (u1)--(u2) (u1)--(u6) (u1)--(u7);
            \draw (u0)--(u8);
            \draw (u2)--(u3) (u3)--(u4) (u4)--(u5);
            \draw (u8)--(u9) (u8)--(u14) (u8)--(u16);
            \draw (u9)--(u10) (u9)--(u12) (u9)--(u13);
            \draw (u10)--(u11);
            \draw (u14)--(u15);
            
            \node[draw=none, fill=none, below=2cm] at (0,-1.5) {$T_2$};
        \end{scope}
    \end{tikzpicture}
    \caption{The smallest non-isomorphic trees with $n=17$ vertices that are cospectral under $R_T$.}
    \label{fig:n17_cospectral}
\end{figure}

\end{example}

\subsection{The sign of $\lambda_{\max}(R_T)$}
\begin{example}[Double-star trees]
\label{ex:double-star}

Let $D_{m,n}$ be the tree consisting of a single edge $\{u,v\}$, where
$d_u=m+1$, $d_v=n+1$, and $u$ (resp. $v$) is adjacent to $m$ (resp. $n$) leaves.

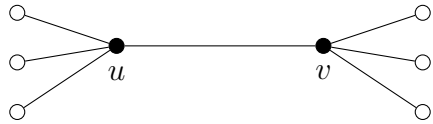
\begin{center}
\begin{tikzpicture}[
    scale=1.1,
    every node/.style={circle, draw, inner sep=2pt},
    center/.style={circle, draw, fill=black, inner sep=2pt}
]

\node[center,label=below:$u$] (u) at (0,0) {};
\node[center,label=below:$v$] (v) at (2.5,0) {};
\draw (u)--(v);

\foreach \i in {1,...,3} {
    \node (ul\i) at (-1.2, {1.0 - 0.6*\i}) {};
    \draw (u)--(ul\i);
}
\foreach \i in {1,...,3} {
    \node (vr\i) at (3.7, {1.0 - 0.6*\i}) {};
    \draw (v)--(vr\i);
}

\end{tikzpicture}
\end{center}

By symmetry, all leaf edges at $u$ (resp. $v$) have the same weight. Let the central edge have weight $z$, and leaf-edge weights be symmetric. The Perron eigenvalue $\lambda$ satisfies:
\begin{equation*}
\lambda = \frac{m - \lambda(m+1) - 3}{(m+1)(\lambda(m+1) + 3)} + \frac{n - \lambda(n+1) - 3}{(n+1)(\lambda(n+1) + 3)}
\end{equation*}

The Perron eigenvalue is determined by a single rational equation in $(m,n)$, and the sign of $\lambda$ changes according to the balance between $m$ and $n$.
\begin{itemize}
    \item \textbf{($\lambda < 0$):} Occurs when $m, n \in \{1, 2\}$.
    \item \textbf{($\lambda = 0$):} Occurs exactly when $m = n = 3$ or $m=2, n=5$.
    \item \textbf{($\lambda > 0$):} Occurs when $m, n \ge 4$.
\end{itemize}
Note that  $\lambda$ becomes positive when $m$ and $n$ grow beyond $3$, it reaches a peak and then asymptotically decreases toward $0$ as $m$ and $n$ continue to grow.
\end{example}

\begin{proposition}\label{pro:D33k}

Let $D_{3,3}$ be the double-star tree with central edge $uv$, where $d_u=d_v=4$ (each vertex is attached to three leaves).  
For any integer $k \ge 1$, let $D_{3,3}^{(k)}$ be the tree obtained by subdividing the central edge $uv$ into a path of length $k$ (i.e., inserting $k-1$ new degree-$2$ vertices).  

Then the largest eigenvalue $\lambda_{\max}$ of the edge-based Ricci matrix $R(D_{3,3}^{(k)})$ satisfies  

\[
\lambda_{\max}\bigl(D_{3,3}^{(k)}\bigr) = 0 \qquad \text{for all } k \ge 1.
\]

In other words, subdividing the central edge of $D_{3,3}$ does \emph{not} change the zero eigenvalue.
    
\end{proposition}

\begin{figure}[h]
\centering
\begin{tikzpicture}[
    scale=1.1,
    every node/.style={circle, draw, inner sep=2pt},
    center/.style={circle, draw, fill=black, inner sep=2pt}
]

% --- left center u ---
\node[center,label=below:$u$] (u) at (0,0) {};

% --- right center v ---
\node[center,label=below:$v$] (v) at (6,0) {};

% --- path vertices (subdivision) ---
\node (x1) at (1.5,0) {};
\node (x2) at (3,0) {};
\node (x3) at (4.5,0) {};

% --- edges of path ---
\draw (u)--(x1)--(x2)--(x3)--(v);

% --- indicate k ---
\node[draw=none,fill=none] at (3,-0.8) {$k\ \text{edges}$};

% --- leaves at u ---
\foreach \i in {1,2,3} {
    \node (ul\i) at (-1.3, {1.2 - 0.7*\i}) {};
    \draw (u)--(ul\i);
}

% --- leaves at v ---
\foreach \i in {1,2,3} {
    \node (vr\i) at (7.3, {1.2 - 0.7*\i}) {};
    \draw (v)--(vr\i);
}

\end{tikzpicture}

\caption{The tree $D_{3,3}^{(k)}$: the central edge of $D_{3,3}$ is subdivided into a path of length $k$.}
\label{fig:S33k}
\end{figure}
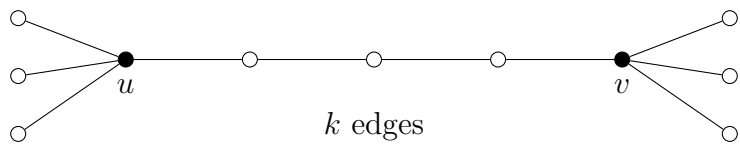

\begin{proof}
We construct a nontrivial vector $w$ on $E(D_{3,3}^{(k)})$ satisfying $Rw=0$.

By symmetry, all leaf edges at $u$ and $v$ have the same weight. 
Let the leaf edges have weight $a$, and let the internal edges adjacent to leaves have weight $b$.

For a leaf edge $\ell$ at $u$ (where $d_u=4$), the eigenvalue equation with $\lambda=0$ gives
\[
\frac{(3a+b)-2a}{4}+\frac{a-2a}{1}=0
\;\Longrightarrow\;
\frac{a+b}{4}-a=0
\;\Longrightarrow\;
b=3a.
\]

Now consider the first internal edge $e_1=(u,x_1)$, and let the next edge have weight $c$. 
Using $b=3a$, the eigenvalue equation yields
\[
\frac{(3a+b)-2b}{4}+\frac{(b+c)-2b}{2}=0
\;\Longrightarrow\;
\frac{3a-b}{4}+\frac{c-b}{2}=0,
\]
which implies $c=b$. Repeating the same argument along the subdivided path shows that all internal edges have weight $b=3a$.

The same computation at $v$ gives leaf weights $a$ there as well. 
Thus we obtain a nontrivial vector $w$ such that $Rw=0$, and hence $0$ is an eigenvalue of $R(D_{3,3}^{(k)})$ for all $k\ge1$.

Since $D_{3,3}$ satisfies $\lambda_{\max}=0$ (Example~\ref{ex:double-star}), 
and the above construction preserves the Rayleigh quotient under subdivision, we conclude
\[
\lambda_{\max}(D_{3,3}^{(k)})=0.
\]
\end{proof}

\begin{remark}
The above result shows that the $(3,3)$ double-star belongs to an infinite family of zero-curvature trees obtained by arbitrarily subdividing the central edge. This is in contrast to asymmetric double-stars (e.g., $S_{2,5}$), where subdivision changes the sign of $\lambda_{\max}$.
\end{remark}

\begin{example}[Trees $T_{m,k}$ with mixed depth-2 branching]
\label{ex:Tmk}

Let $T_{m,k}$ consist of a central vertex $c$ with:
\begin{itemize}
    \item $m$ length-2 branches $c-a_i-a_i'$,
    \item $k$ leaves $d_1,\dots,d_k$ attached directly to $c$.
\end{itemize}

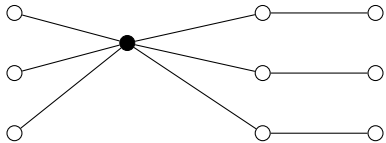
\begin{center}
\begin{tikzpicture}[
    scale=1,
    every node/.style={circle, draw, inner sep=2pt},
    center/.style={circle, draw, fill=black, inner sep=2pt}
]

\node[center] (c) at (0,0) {};

\foreach \i in {1,2,3} {
    \node (a\i) at (1.8, {1.2 - 0.8*\i}) {};
    \node (ap\i) at (3.3, {1.2 - 0.8*\i}) {};
    \draw (c)--(a\i)--(ap\i);
}

\foreach \j in {1,2,3} {
    \node (d\j) at (-1.5, {1.2 - 0.8*\j}) {};
    \draw (c)--(d\j);
}

\end{tikzpicture}
\end{center}

By symmetry, edges fall into three classes: $c\!-\!a_i$, $a_i\!-\!a_i'$, and $c\!-\!d_j$. Let their weights be $p,q,r$ respectively. The eigenvalue equations reduce to a coupled nonlinear system which can be eliminated to obtain a single equation for $\lambda_{m,k}$.

The sign of $\lambda$ is determined as follows:

\begin{center}
\begin{tabular}{|c|c|c|}
\hline
$m$ & $\lambda < 0$ & $\lambda = 0$ \\
\hline
$2$ & $k \le 3$ & $k = 4$ \\
$3$ & none & $k = 0$ \\
$\ge 4$ & none & none \\
\hline
\end{tabular}
\end{center}

In particular:
\begin{itemize}
    \item For $m=2$, $k=0$: $T_{2,0}$ is the path $P_5$, $\lambda = -1 + \frac{\sqrt{2}}{2} \approx -0.293$.
    \item For $m=2$, $k=3$: $T_{2,3}$ has $\lambda < 0$ (approximately $-0.069$).
    \item For $m=2$, $k=4$: $T_{2,4}$ has $\lambda = 0$.
    \item For $m=3$, $k=0$: $T_{3,0}$ is the tree $S_3^2$, $\lambda = 0$.
    \item For $m=3$, $k \ge 1$: $\lambda > 0$.
    \item For $m \ge 4$: $\lambda > 0$ for all $k \ge 0$.
\end{itemize}

Observe that when $m=2$, the tree $T_{2,k}$ is a caterpillar (after removing leaves, the remaining graph is the path $c-a_1-a_2$). When $m \ge 3$, the tree is not a caterpillar (removing leaves leaves a star with center $c$ and $m$ leaves $a_i$), and in these cases $\lambda \ge 0$. 

\end{example}

\subsection{Counterexample: Global minimum on internal edges for $\lambda > 0$}

We now present an explicit counterexample showing that when $\lambda > 0$, the global minimum of the Perron vector need not occur at a leaf edge. 

Let $D^k_{m,n}$ denote the tree obtained from the double-star tree $D_{m,n}$ by subdividing the central edge $\{u,v\}$ into a path of length $k+1$ (i.e., inserting $k$ new degree-$2$ vertices between $u$ and $v$).

\begin{figure}[htbp]
\centering
\begin{tikzpicture}[
    scale=1.1,
    every node/.style={circle, draw, inner sep=2pt},
    center/.style={circle, draw, fill=black, inner sep=2pt}
]

% --- left center u ---
\node[center,label=below:$u$] (u) at (0,0) {};

% --- right center v ---
\node[center,label=below:$v$] (v) at (8,0) {};

% --- left path vertices ---
\node (x1) at (1.5,0) {};
\node (x2) at (2.7,0) {};

% --- minimum edge vertices (blue, solid) ---
\node (m1) at (3.9,0) {};
\node (m2) at (5.1,0) {};

% --- right path vertices ---
\node (x3) at (6.3,0) {};
\node (x4) at (7.5,0) {};

% --- edges ---
\draw[thick] (u)--(x1);
\draw[thick] (x1)--(x2);
\draw[thick, densely dashed] (x2)--(m1);
\draw[thick, blue] (m1)--(m2);  % minimum edge in blue solid
\draw[thick, densely dashed] (m2)--(x3);
\draw[thick] (x3)--(x4);
\draw[thick] (x4)--(v);

% --- indicate k ---
\node[draw=none,fill=none] at (4.5, -0.7) {$k=29$ (path of length $30$)};

% --- leaves at u (m=4) ---
\foreach \i in {1,2,3,4} {
    \node (ul\i) at (-1.1, {1.0 - 0.5*\i}) {};
    \draw[thick] (u)--(ul\i);
}

% --- leaves at v (n=4) ---
\foreach \i in {1,2,3,4} {
    \node (vr\i) at (9.1, {1.0 - 0.5*\i}) {};
    \draw[thick] (v)--(vr\i);
}

% --- weight labels ---
\node[draw=none,fill=none, font=\tiny] at (0.75, 0.35) {$0.302$};
\node[draw=none,fill=none, font=\tiny] at (2.1, 0.35)  {$0.239$};
\node[draw=none,fill=none, font=\tiny, blue] at (4.5, 0.35)  {$0.091$};
\node[draw=none,fill=none, font=\tiny] at (6.9, 0.35)  {$0.239$};
\node[draw=none,fill=none, font=\tiny] at (8.25, 0.35) {$0.302$};

\node[draw=none,fill=none, font=\tiny] at (-1.5, 0.2)  {$0.099$};
\node[draw=none,fill=none, font=\tiny] at (9.5, 0.2)   {$0.099$};

\end{tikzpicture}

\caption{The tree $D^{29}_{4,4}$ with raw edge weights. Dashed lines indicate the omitted path edges. 
Leaf edges have weight $0.099$, strictly lighter than adjacent internal edges ($0.302$), consistent with Corollary~\ref{cor:leaf-local}(2). 
The global minimum $0.091$ (blue) lies on a central internal edge, showing that for $\lambda>0$ the minimum need not occur at a leaf.}
\label{fig:D4k29}
\end{figure}
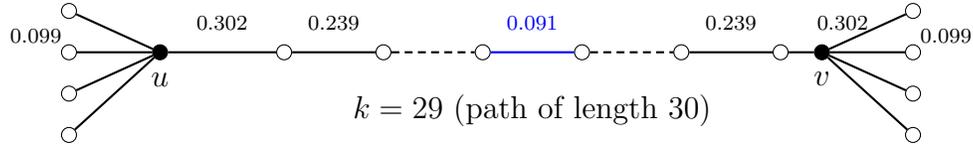

\begin{example}[$m=n=4$, $k=29$]
\label{ex:D4k29}
Consider $D^{29}_{4,4}$, which consists of two $5$-stars connected by a path of length $30$. Numerical computation yields
\[
\lambda_{\max} \approx 0.00774 > 0.
\]
The unnormalized edge weights are:
\begin{itemize}
    \item Leaf edges ($8$ edges): $0.09947$.
    \item Path edge adjacent to the left center, $(0,1)$: $0.30226$ (global maximum).
    \item Central path edge $(15,16)$: $0.09145$ (global minimum).
\end{itemize}
Consistent with Corollary~\ref{cor:leaf-local}(2), leaf edges ($0.09947$) are strictly lighter than their adjacent internal edges ($0.30226$). Nevertheless, the central path edge ($0.09145$) is strictly lighter than the leaf edges, so the global minimum lies on an internal edge.

Similar behavior occurs for other parameters:
\begin{itemize}
    \item $D^{14}_{9,9}$: $\lambda_{\max} \approx 0.0419$, leaf weight $0.11565$, min path weight $0.09072$.
    \item $D^{14}_{19,19}$: $\lambda_{\max} \approx 0.0515$, leaf weight $0.09539$, min path weight $0.07024$.
\end{itemize}
\end{example}

\appendix
\section{Counterexamples  to spectral monotonicity}\label{sec:appendix}

This appendix provides explicit examples showing that the largest eigenvalue of the Ricci matrix is not monotone under two natural tree operations: subdivision and leaf attachment.

\begin{example}
Original edge set:
\[
\begin{aligned}
&(0,1),(0,6),(0,10),(0,14),\ (1,2),(1,4),(1,5),(2,3),\\
&(6,7),(6,8),(6,9),\ (10,11),(10,12),(10,13),\ (14,15),(14,16),
\end{aligned}
\]

Attach a leaf to vertex $0$ of degree $4$.

Eigenvalues:
\[
\lambda_{\mathrm{orig}} = 0.40518747,\quad
\lambda_{\mathrm{new}} = 0.40513310,
\]
\[
\Delta = -5.44\times 10^{-5}.
\]
\end{example}

\begin{example}
Original tree:
\[
(1,5),\ (0,6),\ (0,4),\ (3,4),\ (3,7),\ (2,3),\ (1,2),\ (1,8),\ (8,9).
\]

After subdividing edge $(2,3)$ by inserting vertex $10$:
\[
(1,5),\ (0,6),\ (0,4),\ (3,4),\ (3,7),\ (1,2),\ (1,8),\ (8,9),\ (2,10),\ (10,3).
\]

Eigenvalues:
\[
\lambda_{\mathrm{orig}} = 0.0474186669,\quad
\lambda_{\mathrm{new}} = 0.0382151329,
\]
\[
\Delta = -0.0092035340.
\]

After subdividing edge $(1,2)$:
\[
(1,5),\ (0,6),\ (0,4),\ (3,4),\ (3,7),\ (2,3),\ (1,8),\ (8,9),\ (1,10),\ (10,2).
\]

Eigenvalues:
\[
\lambda_{\mathrm{orig}} = 0.0474186669,\quad
\lambda_{\mathrm{new}} = 0.0382151329,
\]
\[
\Delta = -0.0092035340.
\]

In this example, 
the original tree is the same, but the subdivided edges
$(2,3)$ and $(1,2)$ are \textbf{not} symmetric. 
Nevertheless, the two subdivided trees  are \textbf{isomorphic}, thus, they have identical  $\lambda_{\max}$.
\end{example}

\section{Open Problems}
\label{sec:openproblems}

Several questions arising from this work remain open for future investigation.

\begin{enumerate}

\item \textbf{Uniqueness of the Perron eigenpair.} 
Does the Perron eigenpair $(\lambda, w)$ of $R_T$ uniquely determine the tree structure $T$? 
Example~\ref{ex:cospectral-pair} shows that non-isomorphic trees can share the same full spectrum of $R_T$, but it remains unknown whether the additional data of the Perron eigenvector (beyond the eigenvalue) suffices for reconstruction. 
If not, what is the maximal set of trees sharing the same Perron data?

\item \textbf{Classification of trees with $\lambda \le 0$.} 
Theorem~\ref{the:lambda-negative-caterpillar} shows that trees with $\lambda < 0$ as caterpillars, but the converse is false. 
What structural property characterizes trees satisfying $\lambda \le 0$?

\item \textbf{Connection to other edge-based matrices.} 
Can $R_T$ be related to other edge-based operators, such as the edge adjacency matrix of the line graph $L(T)$ or the edge Laplacian? 
The matrix $R_T$ has the form 
\[
R_T = D - A_{\text{edge}},
\]
where $D$ is a diagonal matrix of edge ``degree-like'' terms and $A_{\text{edge}}$ encodes edge incidences with degree weights. 
Exploring such connections might yield new spectral invariants for trees and shed light on the geometric meaning of the Perron eigenpair.

\item \textbf{Monotonicity and extremal problems.} 

We hypothesis that
attaching a leaf at a vertex of degree $3$ never decreases $\lambda_{\max}(R_T)$; that is,
the case $d=3$ of Proposition~\ref{pro:attachment} holds without any restriction on
$\lambda_{\max}(R_T)$. Consequently, $\lambda_{\max}(R_T)$ can decrease under leaf
attachment only at a vertex of degree $\ge 4$.

\end{enumerate}

{\bf Acknowledgements:} This work was supported by the National Natural Science Foundation of China [grant numbers 12301434 to S. Bai, 12371056 to B. Hua].

 \bibliographystyle{plain}
 \bibliography{refs}
\end{document}